\documentclass[a4paper,12pt]{article}
\usepackage{amsfonts}
\usepackage{latexsym}
\usepackage{amsmath}
\usepackage{amssymb}
\usepackage{amssymb}

\usepackage{slashed}
\usepackage{upgreek}
\usepackage{mathrsfs}
\usepackage{bbm}
\usepackage{enumerate}

\usepackage[toc,page]{appendix}

\usepackage{amsthm}
\usepackage[utf8]{inputenc}
\usepackage{graphicx}

\theoremstyle{remark}
\newtheorem*{remark}{Remark}

\usepackage[utf8]{inputenc}
\usepackage[english]{babel}

\theoremstyle{definition}
\newtheorem{definition}{Definition}[section]

\newtheorem{theorem}{Theorem}[section]

\newtheorem{lemma}[theorem]{Lemma}

\newcommand*{\defeq}{\mathrel{\vcenter{\baselineskip0.5ex \lineskiplimit0pt
			\hbox{\scriptsize.}\hbox{\scriptsize.}}}%
	=}

\allowdisplaybreaks

\usepackage{relsize}
\usepackage{graphicx}
\usepackage{comment}

\renewcommand{\thefootnote}{\fnsymbol{footnote}}

\def\appendix#1{\addtocounter{section}{1}\setcounter{equation}{0}
	\renewcommand{\thesection}{\Alph{section}}
	\section*{Appendix \thesection\protect\indent \parbox[t]{11.15cm}{#1}}
	\addcontentsline{toc}{section}{Appendix \thesection\ \ \ #1}}

\newcommand{\pp}{=\kern-0.40em{\vert}}

\def\bbe{{\bf{e}}}

\font\mybb=msbm10 at 11pt

\def\bb#1{\hbox{\mybb#1}}

\def\bZ {\bb{Z}}
\def\bR {\bb{R}}

\def\bC {\bb{C}}

\def\vkappa{{\vec{\kappa}}}

\newcommand{\bea}{\begin{eqnarray}}
	\newcommand{\eea}{\end{eqnarray}}

\usepackage[usenames,dvipsnames,svgnames,table]{xcolor}	
\usepackage[normalem]{ulem}				
\usepackage{microtype}
\usepackage[colorlinks, citecolor=blue, linkcolor=blue, urlcolor=Maroon, filecolor=Maroon, linktocpage=true]{hyperref} 

\usepackage{chngcntr}
\counterwithout{equation}{section}

\begin{document}

	\begin{center}
		\vspace*{-1.0cm}
		\begin{flushright}
		\end{flushright}

		
		\vspace{2.0cm} {\Large \bf  Derivations, holonomy groups and heterotic  geometry} \\[.2cm]
		
		\vskip 2cm
		 G.\,  Papadopoulos
		\\
		\vskip .6cm


		\begin{small}
			\textit{Department of Mathematics
				\\
				King's College London
				\\
				Strand
				\\
				London WC2R 2LS, UK}
			\\*[.3cm]
			\texttt{george.papadopoulos@kcl.ac.uk}
		\end{small}
		\\*[.6cm]

	\end{center}

	\vskip 2.5 cm

	\begin{abstract}
\noindent

	We investigate the superalgebra of derivations generated by the fundamental forms on manifolds with reduced structure group. In particular, we point out a relation between the algebra of derivations of heterotic geometries that admit Killing spinors and the commutator algebra  of holonomy symmetries in sigma models. We use this to propose a Lie bracket on the space of fundamental forms of all heterotic geometries  with a non-compact holonomy group and present the associated derivation algebras. We also explore the extension of these results to heterotic geometries  with compact holonomy groups and,  more generally, to manifolds with reduced structure group. A brief review of the classification of heterotic geometries that admit Killing spinors and an extension of this classification  to  some heterotic inspired geometries are also included.
	
	\end{abstract}

	

	\newpage
	
	\renewcommand{\thefootnote}{\arabic{footnote}}

\section{Introduction}

The set of derivations, $\mathfrak{Der}\big(\Omega^*({\mathcal M})\big)$, on the space of forms, $\Omega^*({\mathcal M})$, generated by vector valued forms, $\vec\Omega^*({\mathcal M})$, of a manifold ${\mathcal M}$ is a $\bZ$-graded superalgebra, see e.g. \cite{phgpw2}. As vector valued forms generate inner  and exterior derivations on  $\Omega^*({\mathcal M})$, these in turn induce   superalgebra structures on  $\vec\Omega^*({\mathcal M})$. In the former case, the superalgebra bracket is constructed using an inner derivation operation in $\vec\Omega^*({\mathcal M})$ while in the latter case the superalgebra bracket is given by the Nijenhuis tensor, $N(\cdot, \cdot)$, of two vector forms.  The Jacobi identities are satisfied as a consequence of those of $\mathfrak{Der}\big(\Omega^*({\mathcal M})\big)$.

On manifolds equipped with a metric $g$, either with Lorentzian or with Euclidean signature\footnote{The construction works with metrics of any signature but the focus will be on the Lorentzian and Euclidean signature manifolds.}, the superalgebra structures\footnote{From now on,  the labelling of various spaces involved with the underlying manifold may be suppressed at convenience, i.e we will denote $\vec\Omega^*({\mathcal M})$ with $\vec\Omega^*$ and so on.} on $\vec\Omega^*$ associated to the inner and exterior derivations can be induced on $\Omega^*$.  Indeed given a $k$-form $\chi\in \Omega^k$, one can construct a vector $(k-1)$-form, $\vec\chi$, as $g(Y, \vec\chi(X_1, \dots, X_{k-1}))\defeq i_Y\chi( X_1, \dots, X_{k-1})$, where $Y, X_1, \dots, X_{k-1}$ are vector fields on ${\mathcal M}$ and $i_Y$ is the inner derivation of $\chi$ with respect to  $Y$.
Denoting the brackets induced on $\Omega^*$ with $\cdot\bar\wedge \cdot$ and $N(\cdot, \cdot)$ associated to the inner and exterior derivations, respectively, the superalgebras $\big(\Omega^*, \cdot\bar\wedge \cdot\big)$ and $\big(\Omega^*, N(\cdot, \cdot)\big)$   are clearly infinite dimensional.

A refinement\footnote{This refinement has also applications in physics,  in particular string theory, and a brief description will be given below and in section \ref{sec:holsym}.} of the constructions above is possible for any manifold ${\mathcal M}$ equipped with a metric, $g$, which in addition exhibits a reduction of its structure group to a subgroup, ${\mathcal H}$, of the appropriate orthogonal group. Denoting with, $\nabla^{\mathcal H}$, a compatible connection whose holonomy group is ${\mathcal H}$, one can consider the set all the forms, $\Omega^*_{\nabla^{\mathcal H}}$,  which are $\nabla^{\mathcal H}$-covariantly constant\footnote{Similarly, we write $\Omega^*_{\nabla^{\mathcal H}}$ instead of $\Omega^*_{\nabla^{\mathcal H}}(\mathcal{M})$  unless the reference to the underlying manifold is necessary to avoid confusion.}.  It turns out that $\big(\Omega^*_{\nabla^{\mathcal H}}, \cdot\bar\wedge \cdot\big)$ is a superalgebra with respect to the bracket induced by inner derivations -- this is because the $\cdot\bar\wedge \cdot$-bracket of two $\nabla^{\mathcal H}$-covariantly constant forms is also $\nabla^{\mathcal H}$-covariantly constant. The superalgebra structure on $\big(\Omega^*_{\nabla^{\mathcal H}}, \cdot\bar\wedge \cdot\big)$ is universal in the sense that it is the same for all Gray-Hervella  \cite{GH}  classes of the ${\mathcal H}$-structure.

Unlike the superalgebra structure on $\Omega^*_{\nabla^{\mathcal H}}$ induced by inner derivations, there are some difficulties inducing a superalgebra structure on $\Omega^*_{\nabla^{\mathcal H}}$ using exterior derivations. One of them is closure. It is not a priori the case that the Nijenhuis tensor of two $\nabla^{\mathcal H}$-covariantly constant form is $\nabla^{\mathcal H}$-covariantly constant. There are also indications that in some cases closure will require an extension to include elements of $\vec\Omega^*_{\nabla^{\mathcal H}}$ that cannot be written as forms.  Note that  if closure holds and
$\big(\Omega^*_{\nabla^{\mathcal H}}, N(\cdot, \cdot)\big)$ is a superalgebra, it will depend on the Gray-Hervella  classes of the ${\mathcal H}$-structure.

Before we proceed further, it is more convenient to investigate the superalgebra structure of a smaller set than $\Omega^*_{\nabla^{\mathcal H}}$. As the wedge product of two $\nabla^{\mathcal H}$-covariantly constant forms is also $\nabla^{\mathcal H}$-covariantly constant, one can consider the subset of {\sl fundamental forms}  of ${\mathcal H}$, $\mathfrak{f}^{\mathcal H}$, in $\Omega^*_{\nabla^{\mathcal H}}$, see definition \ref{fforms}. Then,  $\Omega^*_{\nabla^{\mathcal H}}$ is generated as a ring from $\mathfrak{f}^{\mathcal H}$  with multiplication the wedge product.

One of the purposes of this paper is to explore the superalgebra structures on $\mathfrak{f}^{\mathcal H}$ in a number of examples that involve manifolds with reduced structure group and with either Euclidean or Lorentzian signature metrics.  The main focus will be on the superalgebra structure on $\mathfrak{f}^{\mathcal H}$ induced by the inner and exterior derivations. The closure, $\mathfrak{f}^{\mathcal H}_{\bar\wedge}$, of $\mathfrak{f}^{\mathcal H}$ under inner derivations $\bar\wedge$ can be easily determined and its properties follow from those of $\big(\Omega^*_{\nabla^{\mathcal H}}, \cdot\bar\wedge \cdot\big)$. The closure,  $\mathfrak{f}^{\mathcal H}_{N}$, of $\mathfrak{f}^{\mathcal H}$ under exterior derivations requires further exploration as that for $\big(\Omega^*_{\nabla^{\mathcal H}}, N(\cdot, \cdot)\big)$. One direction to proceed is to embark into a systematic investigation of Gray-Hervella classes of ${\mathcal H}$-structure and find those that imply the closure of $\mathfrak{f}^{\mathcal H}_{N}$.

Further progress can be made after a detailed investigation of the holonomy symmetries of sigma models with target spaces
heterotic (inspired) geometries. These symmetries have been introduced in the context of string theory \cite{phgpw1, phgpw2} following earlier work in \cite{odake, dvn}.  The heterotic (inspired) geometries are equipped with a metric connection, $\hat\nabla$, with skew-symmetric torsion, $H$, and have extensively been investigated both in physics and mathematics literature; for some selective publications see  \cite{cmhew, strominger, hkt, poon,  stefang2,  salamon1, gauntlett, lustb, green}.  The holonomy group, ${\mathcal H}$, of $\hat\nabla$ can be either compact or non-compact.  The reason that these are useful in the current context is that the commutator of two holonomy symmetries, which are generated by elements of  $\mathfrak{f}^{\mathcal H}$, is also a symmetry of the sigma model and so closure in guaranteed.
Moreover, in the commutator of two holonomy symmetries generated by  $\phi_1, \phi_2\in \mathfrak{f}^{\mathcal H}$, the Nijenhuis tensor, $N(\vec\phi_1, \vec\phi_2)$, appears. As a result, there is a relation between the commutator of two holonomy symmetries and that of the exterior derivations generated by $\vec\phi_1$ and $\vec\phi_2$. Following this paradigm, we find that the commutator of two holonomy symmetries generated by the fundamental forms of a non-compact holonomy group is determined in terms of a new inner derivation operation on $\mathfrak{f}^{\mathcal H}$, $\bar\curlywedge$, see theorem \ref{barcwedge} and \cite{gpjp1}. This leads  to the new algebra $\mathfrak{f}^{\mathcal H}_{\bar\curlywedge}$.  We determine the $\mathfrak{f}^{\mathcal H}_{\bar\curlywedge}$ algebras for all heterotic geometries as well as some heterotic inspired ones with non-compact holonomy groups.  The commutators of two holonomy symmetries generated by elements of $\mathfrak{f}^{\mathcal H}$, for ${\mathcal H}$ compact, are also explored. We indicate how the closure of this commutator  can be used to define a new bracket  on an appropriate extension of $\mathfrak{f}^{\mathcal H}$. This may also be extended to other $\nabla^{\mathcal H}$ connections.

The main new results described  this paper are presented in theorems   \ref{thm3}, \ref{thm2}, \ref{prin2}, \ref{ncomalg1} and \ref{ncomalg2}.  Some other key results have also been reviewed. These include those described in the theorems  \ref{ncom}, \ref{thm1}, \ref{thlms} and \ref{ncomcom}. The classification of heterotic geometries that admit Killing spinors mentioned in theorems  \ref{ncom} and  \ref{thm1} has been both refined and simplified to make it  suitable for the purpose of this work.

As in what follows we shall extensively use  the geometry of  heterotic backgrounds which admit Killing spinors, i.e. supersymmetric heterotic backgrounds,  we have included a brief description of their classification.    This remains the only theory in ten dimensions that such a classification has been achieved.  The main reason for this is that the parallel transport equation of the Killing spinor equations  (KSEs) is associated with the connection $\hat\nabla$.  Moreover,   the spinor bundle is associated to  the 16-dimensional Majorana-Weyl representation of $\mathrm{Spin}(9,1)$ that is the double cover of $SO(9,1)$.  This is a chiral real representation of $\mathrm{Spin}(9,1)$. As a result, the Killing spinors, which are parallel with respect to $\hat\nabla$, can be identified using group representation theory. In turn,  the holonomy group of $\hat\nabla$ can be determined in each case as the  isotropy group of parallel spinors in $\mathrm{Spin}(9,1)$.  After identifying the $\hat\nabla$-parallel spinors, i.e. the solutions of the gravitino KSE, one can proceed to solve the remaining dilatino and gaugino KSEs.  We also present an extension of the classification mentioned above to heterotic inspired geometries in eight dimensions.

This paper is organised as follows. In section 2, we shall summarise the main properties in the theory of derivations that we shall explore later, and introduce a new inner derivation $\bar\curlywedge$. In section 3, we shall summarise the classification of the geometry of supersymmetric heterotic  backgrounds. In section 4, we extend the classification results of section 3 to some heterotic inspired geometries. In section five, we present the commutator of two holonomy symmetries for 1-dimensional sigma models and present the algebra of fundamental forms, $\mathfrak{f}^{\mathcal H}_{\bar\curlywedge}$, for heterotic and heterotic inspired geometries with non-compact holonomy groups. In section 6, we give our conclusions.

\section{Derivations}

\subsection{Derivations and superalgebras}

Let $D_\ell$ be a derivation of degree $\ell$ on the space of forms, $\Omega^*=\Omega^*(\mathcal{M})$, of a manifold ${\mathcal M}$, i.e. $D_\ell: \Omega^p\rightarrow \Omega^{p+\ell}$, such that
\begin{enumerate}[(i)]

\item   $D_\ell(a\psi+ b\chi)=a D_\ell\psi+b D_\ell\chi$,   $\forall a,b\in \bR$ and $\psi, \chi\in \Omega^*$, and

 \item $D_\ell(\psi\wedge \chi)=D_\ell\psi\wedge \chi+(-1)^{p\ell} \psi\wedge D_\ell\chi$~, where $\psi\in \Omega^p$ ~ {\small (Leibnitz property)}.

 \end{enumerate}
 Given two derivations $D_\ell, D_r\in {\mathfrak Der}(\Omega^*)$, one can define a commutator
\bea
[D_\ell, D_r]\defeq D_\ell D_r-(-1)^{r\ell} D_r D_\ell~.
 \eea
 This commutator satisfies the (super-)Jacobi identity and turns ${\mathfrak Der}(\Omega^*)$ into an infinite dimensional $\bZ$-graded superalgebra.

Given a vector(-valued) $\ell$-form $L\in \vec\Omega^{\ell}$, one can define two  derivations.  One is the inner derivation $i_L$.  This is defined as $i_L f=0$ and $i_L\omega=\omega(L)$ on the space of 0- and 1-degree forms, respectively. Then it can be easily extended\footnote{In a coordinate basis, $i_L\phi={1\over \ell! (p-1)!} L^\mu{}_{\nu_1\dots \nu_\ell} \phi_{\mu \nu_{\ell+1}\dots \nu_{\ell+p-1}} dx^{\nu_1\dots \nu_{\ell+p-1}}$, where $\phi\in \Omega^p$ and we have used a shorthand notation $dx^{\nu_1\dots \nu_{\ell+p-1}}=dx^{\nu_1}\wedge\dots \wedge dx^{\nu_{\ell+p-1}}$ to denote the standard coordinate basis in the space of forms.} to $\Omega^*$ using the Leibnitz property. The other is
$d_L\defeq i_L d+ (-1)^\ell d i_L$, where $d$ is the usual exterior derivative on $\Omega^*$. Clearly, $d_L$ is a generalisation of the Lie derivative along a  vector field. It is also a generalisation of $d$ for which $L$ is the identity vector 1-form, i.e $L(X)=X$ for every vector field $X$ on ${\mathcal M}$.   It turns out that the commutators of these derivations, see e.g. \cite{phgpw2}, are
\bea
&&[d_L, d_M]= d_{N(L, M)}~,~~~[i_L, d_M]=d_{L\barwedge M}+(-1)^m i_{N(L,M)}~,~~~
\cr
&&[i_L, i_M]=i_{L\barwedge M}+(-1)^{\ell+m+m\ell} i_{M\barwedge L}~,
\label{exincom}
\eea
where $N(L, M)$ is a generalisation of the Nijenhuis tensor for almost complex structures and $L\barwedge M\in \vec\Omega^{\ell+m-1}$ such that $\omega(L\barwedge M)\defeq i_L (\omega(M))$ for every $\omega\in \Omega^1$. An explicit formula for $N(L, M)$ is given in section \ref{sec:holsym}. Notice that $d d_L=(-1)^\ell d_L d$.

The exterior derivation map $D: \vec\Omega^{\ell}\rightarrow {\mathfrak Der}(\Omega^*)$, which associates $L$ to $d_L$,  is a linear map and 1-1. Requiring that $D$ is a Lie superalgebra homomorphism, the commutator  of two exterior derivatives in (\ref{exincom}) can be used to  introduce a bracket on $\vec\Omega^*$ given by $N(\cdot, \cdot)$. The inner derivation map $i: \vec\Omega^{\ell}\rightarrow {\mathfrak Der}(\Omega^*)$, which associates $L$ to $i_L$, is also a linear map and 1-1. Therefore, one can associate a bracket on $\vec\Omega^*$ given by the right hand side of the last equation in (\ref{exincom}).

Next suppose that $L=\vec\lambda$ and $M=\vec \chi$, where $\lambda\in \Omega^{\ell+1}$ and $\chi\in \Omega^{m+1}$. In such a case, the last equation in (\ref{exincom}) becomes
 \bea
 [i_{\vec \lambda}, i_{\vec\chi}]=(-1)^{\ell+1} i_{\overrightarrow{\lambda\barwedge \chi}}~,
 \eea
 where $\lambda\barwedge \chi\defeq i_{\vec\lambda}\chi$.
As a result, the inner derivations induce a bracket $\cdot\bar\wedge\cdot$ on $\Omega^*$ associated with the  operation $\barwedge$. This  turns $\Omega^*$ into a infinite dimensional superalgebra.

 Suppose that a manifold ${\mathcal M}^n$ admits a reduction of its structure group to a subgroup ${\mathcal H}$ of $SO_n$, where $SO_n$ stands for either $SO(n)$ or $SO(n-1,1)$.

  \begin{definition}\label{fforms}
$\mathfrak{f}^{\mathcal H}$ is the set of fundamental forms of the  ${\mathcal H}$-structure, iff the $\nabla^{\mathcal H}$-covariant constancy of the elements of $\mathfrak{f}^{\mathcal H}$
imply that the holonomy group of $\nabla^{\mathcal H}$ is  ${\mathcal H}$ and  $\Omega^*_{\nabla^{\mathcal H}}$ is generated as a ring from $\mathfrak{f}^{\mathcal H}$ with multiplication the wedge product. \hfill  	$\vartriangle$
\end{definition}

   As $\barwedge$-product of two $\nabla^{\mathcal H}$-covariantly constant forms is covariantly constant, one can consider the closure of the set $\mathfrak{f}^{\mathcal H}$ under the bracket $\cdot\barwedge\cdot$ operation. This will yield the (super){\it algebra} of fundamental forms of the holonomy group ${\mathcal H}$ that we shall  denote with $\mathfrak{f}_{\bar\wedge}^{\mathcal H}$. Typically, $\mathfrak{f}^{\mathcal H}$ is included in $\mathfrak{f}_{\bar\wedge}^{\mathcal H}$ as the latter may require elements from $\Omega^*_{\nabla^{\mathcal H}}$  for closure, which are not included in $\mathfrak{f}^{\mathcal H}$.  This will be demonstrated by explicit examples below. As it has already been mentioned, the superalgebra structure on $\mathfrak{f}_{\bar\wedge}^{\mathcal H}$  is universal in the sense that it will be the same for all Gray-Hervella classes of the ${\mathcal H}$-structure.

 One can also consider the closure properties of  $\mathfrak{f}^{\mathcal H}$ with respect to the Nijenhuis bracket $N(\cdot, \cdot)$  associated with the commutator of two exterior derivations. As it has already been mentioned, it is not immediately the case that  closure holds as the Nijenhuis tensor of two $\nabla^{\mathcal H}$-covariantly constant forms may not be $\nabla^{\mathcal H}$-covariantly constant -- closure can also depend on the choice of connection $\nabla^{\mathcal H}$. Supposing that such a closure is well defined, the superalgebra structure of $\mathfrak{f}^{\mathcal H}_N$
  will  be sensitive to the Gray-Hervella classes of the ${\mathcal H}$-structure.   We shall demonstrate this with examples.

  {\it Example 1:} For all Berger manifolds the Nijenhuis tensors of all fundamental forms vanish. Therefore, $\mathfrak{f}^{\mathcal H}_N$ is abelian.

   {\it Example 2:} Consider a manifold ${\mathcal M}$ equipped with an almost complex structure $I$ and with a compatible metric connection with skew-symmetric torsion $\hat\nabla$, $\hat\nabla I=0$. One can demonstrate that the Nijenhuis tensor $N(I,I)$ of  $I$ is associated to a 3-form and is $\hat\nabla$-covariantly constant \cite{dvn}. The superalgebra $\mathfrak{f}^{\mathcal H}_N$ is not abelian and it  is given by the relations $[d_I, d_I]=d_N$ and $[d_I, d_N]=0$.  The latter commutator is a consequence of the Jacobi identities. The superalgebra is isomorphic to the $N=1$ supersymmetry algebra, $\mathfrak{s}(1)$, in one dimension with $d_N$ identified with the hamiltonian and $d_I$ with the supersymmetry generator, respectively. If ${\mathcal M}$ is 6-dimensional,  ${\mathcal H}$ can be a subgroup of $SU(3)$. An explicit example of a manifold that admits such a structure is the 6-sphere.

\begin{remark}

If one identifies both $\mathfrak{f}^{\mathcal H}_{\bar\wedge}$ and $\mathfrak{f}^{\mathcal H}_N$, it is straightforward to determine the superalgebra of derivations, $\mathfrak{Der}\big(\Omega_{\nabla^{{\mathcal H}}}^*\big)$ generated by the $\nabla^{{\mathcal H}}$-covariantly constant forms.  This is a consequence of the commutators of inner and exterior derivations in (\ref{exincom}).

\end{remark}

\subsection{Inner derivations for null forms}\label{sec:null}

\begin{definition}
Let ${\mathcal M}$ be a Lorentzian manifold with metric $g$ equipped with a null nowhere vanishing 1-form, $\kappa$, $g(\vkappa, \vkappa)=0$.
 A form $\phi$ on ${\mathcal M}$ is null along $\kappa$, iff
\bea
\kappa\wedge \phi=0~,~~~\kappa\bar\wedge\phi=0~.
\label{nullcon}
\eea
Moreover, denote the space of all such forms with $\Omega^*_\kappa({\mathcal M})$. \hfill  	$\vartriangle$
\end{definition}

To describe the elements of $\Omega^*_\kappa({\mathcal M})$, a decomposition is required of the forms along the direction of $\kappa$ and directions  transverse to $\kappa$. However, it is well known feature of Lorentzian geometry that there is an ambiguity in decomposing the tangent space of a Lorentzian manifold in directions along and transverse to a null vector field.  To describe this ambiguity in the present context observe that the existence of a null nowhere vanishing null 1-form $\kappa$ on ${\mathcal M}^n$ reduces the structure from $SO(n-1,1)$ to the subgroup of $SO(n-2)\ltimes \bR^{n-2}$ that stabilises $\kappa$. Then, we introduce a local trivialisation $\{U_\alpha\}_{ \alpha \in I}$ of $T{\mathcal M}^n$ and a local pseudo-orthonormal frame $(\bbe_\alpha^-, \bbe_\alpha^+, \bbe_\alpha^i; i=1,\dots, n-2)$ such that $\bbe_\alpha^-=\kappa_\alpha$ and $g_\alpha=2 \bbe_\alpha^-\bbe_\alpha^++\delta_{ij}\bbe_\alpha^i\bbe_\alpha^j$. Of course $\bbe^-_\alpha=\bbe^-_\beta$ and $g_\alpha=g_\beta$ at the intersection $U_\alpha\cap U_\beta$ of two open sets $U_\alpha$ and $U_\beta$, but
\bea
\bbe^+_\alpha= \bbe_\beta^+-{1\over2} w_{\alpha\beta}^2 \bbe_\beta^-+(w_{\alpha\beta})_i\bbe_\beta^i~,~~~ \bbe_\alpha^i= (O_{\alpha\beta}^{-1})^i{}_j (\bbe_\beta^j- w_{\alpha\beta}^j \bbe_\beta^-)~,
\label{patch}
\eea
where $(O_{\alpha\beta}, w_{\alpha\beta})\in SO(n-2)\ltimes\bR^{n-2}$ parameterise the isotropy subgroup of $SO(n-1,1)$ that stabilises $\kappa$ and depend on the points of $U_\alpha\cap U_\beta$. Clearly if the metric is decomposed along the lightcone and transverse to the lightcone directions on $U_\alpha$, the components of such decomposition are not preserved by the patching condition (\ref{patch}).  The task is to describe the geometry of the spacetime in a way that it is independent from the choice of the decomposition.  Also note that if the structure group further  reduces to ${\mathcal K}\ltimes\bR^{n-2}$, then the patching condition  for the
pseudo-orthonormal frame is as in (\ref{patch}) with $(O_{\alpha\beta}, w_{\alpha\beta})\in {\mathcal K}\ltimes\bR^{n-2}$.

To describe the elements of $\Omega^*_\kappa({\mathcal M}^n)$ on a patch $U_\alpha$ use  the coframe $(\bbe_\alpha^-, \bbe_\alpha^+, \bbe_\alpha^i; i=1,\dots n-2)$ above and  observe that the conditions, $\kappa\wedge \phi=\kappa\bar\wedge \phi=0$, imply that
\bea
\phi_\alpha={1\over \ell!}(\phi_\alpha)_{-i_1\dots i_\ell}\, \bbe^-\wedge \bbe^{i_1}\wedge \dots \wedge \bbe^{i_\ell}\equiv {1\over \ell!}(\phi_\alpha)_{i_1\dots i_\ell}\, \bbe^{-i_1 \dots i_\ell}~,
\label{exL}
\eea
where it has been used that $\kappa=\bbe^-$ is nowhere vanishing.
Therefore, one can associate to every $\phi\in \Omega^*_\kappa({\mathcal M}^n)$ a collection of forms $\{\breve \phi_\alpha\}_{\alpha\in I}$ each defined on the open set $U_\alpha$ of ${\mathcal M}^n$ such that
\bea
\breve \phi_\alpha={1\over \ell!}(\phi_\alpha)_{i_1\dots i_\ell}  \bbe^{i_1\dots i_\ell}~.
\label{tphi}
\eea
  This collection   $\breve \phi=\{\breve \phi_\alpha\}$ is not a globally defined form on ${\mathcal M}^n$. Instead, at the intersection $U_\alpha\cap U_\beta$ of two open sets $U_\alpha$ and $U_\beta$,
\bea
\breve \phi_\alpha= \breve \phi_\beta+ \bbe^-\wedge \breve \chi_{\alpha\beta}~,
\label{patch2}
\eea
where $\breve \chi_{\alpha\beta}={1\over (\ell-1)!}(\chi_{\alpha\beta})_{i_1\dots i_{\ell-1}}  \bbe^{i_1\dots i_{\ell-1}}$. One can establish this using $\phi_\alpha=\bbe^- \wedge \breve \phi_\alpha=\bbe^- \wedge \breve \phi_\beta=\phi_\beta$ on $U_\alpha\cap U_\beta$ and that $\bbe^-$ is nowhere vanishing on ${\mathcal M}^n$. Clearly, there is a 1-1 correspondence between the collections $\breve \phi$ and null forms $\phi$ along $\kappa$. This correspondence is described below as $\phi$ {\sl is  represented} by $\breve\phi$, or equivalently, as $\breve\phi$ {\sl represents} $\phi$.

\begin{remark}
The operation $\kappa\wedge$  is cohomology operation on the space of forms. The computation presented above is based on the local triviality of $\kappa\wedge$.
\end{remark}

\begin{theorem}\label{barcwedge}
Let $\lambda, \phi\in \Omega^*_\kappa({\mathcal M}^n)$.  Define on $U_\alpha$
\bea
{\lambda\bar{\curlywedge} \phi}\defeq \bbe^-\wedge i_{\vec{\breve \lambda}} \breve \phi~,~~~ \lambda\curlywedge \phi\defeq \lambda\wedge \breve \phi~,~~~
\label{twoop}
\eea
where  $\breve \lambda$ and $\breve \phi$ are defined as in (\ref{tphi}), respectively, and we have suppress the open set labelling. Then, both $\lambda\bar{\curlywedge} \phi$ and $\lambda\curlywedge \phi$ are globally defined forms on ${\mathcal M}^n$.

Moreover, if $\kappa$, $\lambda$ and $\phi$ are covariantly constant with respect to a metric connection $\nabla^{\mathcal H}$, i.e. ${\mathcal H}\subseteq SO(n-2)\ltimes \bR^{n-2}$,  then $\lambda\curlywedge \phi$ and $\lambda\bar{\curlywedge} \phi$ are also $\nabla^{\mathcal H}$-covariantly constant.
\end{theorem}
\begin{proof}
The  operations $\curlywedge$ and $\bar\curlywedge$ have arisen after an investigation of the commutator (\ref{walg}) of holonomy symmetries in sigma models, see \cite{gpjp1}.  All statements of the theorem follow after a direct computation. The global definition of $\lambda\bar{\curlywedge} \phi$ and $\lambda\curlywedge \phi$ follows upon applying
the patching condition (\ref{patch}).
\end{proof}

\begin{remark}
Many examples of the Lorentzian geometry as that described above will be given in the context of heterotic geometries, see section \ref{sec:nullhetgeom}.  These are 10-dimensional geometries but they can be easily generalised to all dimensions.  More generally, geometries that admit $\nabla^{\mathcal H}$-covariantly constant forms that are null along $\kappa$ arise in the context of null G-structures \cite{null}, which have recently been investigated in \cite{alekseevsky, fino3, bruce, murcia}.   Such  geometries in all dimensions will be constructed in section \ref{lorexn}.

\end{remark}

\begin{remark}

The existence of the operation $\bar{\curlywedge}$ given in (\ref{twoop}), which can be shown to satisfy the Jacobi identity, allows to turn $\Omega^*_\kappa$ into an (infinite dimensional) superalgebra with bracket $\cdot \bar{\curlywedge} \cdot$.  Note that $\Omega^*_\kappa$ is abelian with respect to the usual bracket $\cdot \bar{\wedge} \cdot$.   Furthermore, as a consequence of theorem \ref{barcwedge}, the space of $\nabla^{\mathcal H}$-covariantly constant forms, which are null along $\kappa$,
$\Omega^*_{\kappa, \nabla^{\mathcal H}}$, is also a superalgebra with bracket $\cdot \bar{\curlywedge} \cdot$. As the $\curlywedge$-product of two $\nabla^{\mathcal H}$-covariantly constant forms is $\nabla^{\mathcal H}$-covariantly constant, one can define the algebra of fundamental forms, $\mathfrak {f}^{\mathcal H}_{\bar{\curlywedge}}$, in analogy  to that of $\mathfrak {f}^{\mathcal H}_{\bar{\wedge}}$.  Then, $\mathfrak {f}^{\mathcal H}_{\bar{\curlywedge}}$ generates $\Omega^*_{\kappa, \nabla^{\mathcal H}}$ as a ring  with multiplication $\curlywedge$. Note that the standard $\wedge$-product of any two elements in $\Omega^*_{\kappa}$ vanishes.

\end{remark}

\section{Heterotic geometry}

\subsection{Preliminaries}

To simplify the analysis that  follows, we make the following assumptions. All manifolds we  consider are smooth, oriented, spin  and simply connected. If the latter is not the case, one can take their universal cover.   Lorentzian signature manifolds are also restricted to be time-oriented and so their structure group reduces to the connected component of the Lorentz group.

As it has already been mentioned, in heterotic geometry, manifolds ${\mathcal M}^n$, with Euclidean  or Lorentzian signature, are equipped  with a metric  $g$ and a 3-form $H$.  As a result, one can define the metric connection, $\hat\nabla$, with skew symmetric torsion, $H$, as
 \bea
\hat\nabla_X Y\defeq\nabla_X Y+{1\over2} \vec H(X,Y)~,
\label{hcon}
\eea
 where $X,Y$ are vector fields on ${\mathcal M}^n$.

 One of the Killing spinor equations that arises in heterotic theory, which can be generalised to any manifold, is the parallel transport equation $\hat\nabla_X\epsilon=0$, where now $\hat\nabla_X$ is the connection induced on the spin bundle $S$ over ${\mathcal M}^n$ and $\epsilon$ is a section of $S$. The typical  representations, $\Delta$, considered in physics associated to $S$ are the irreducible real or complex  spinor representations of $\mathrm{Spin}_n$ as well as some reducible ones that include the Dirac representation.

A priori the holonomy group of $\hat\nabla$ is in $\mathrm{Spin}_n$. However, the existence of parallel spinors $\epsilon$, $\hat\nabla_X\epsilon=0$, will reduce the holonomy group to a subgroup of the isotropy group ${\mathcal H}$ of the parallel spinors in $\mathrm{Spin}_n$.  Therefore, to identify all holonomy groups of manifolds with such parallel spinors, one should find all subgroups ${\mathcal H}$ of $\mathrm{Spin}_n$ such that the module $\Delta$  associated to the spinor bundle, under investigation, includes copies of the trivial module after a decomposition into representations of ${\mathcal H}$.  Each copy of the trivial module of ${\mathcal H}$ in the decomposition of $\Delta$ corresponds to an additional parallel spinor. To do this, one begins with an analysis of the orbits of $\mathrm{Spin}_n$ in $\Delta$. The isotropy group of every orbit will preserve at least one spinor.  This identifies all possible holonomy groups ${\mathcal H}$ that allow for at least one parallel spinor. Then the orbits of ${\mathcal H}$ in $\Delta$ are further examined.  Their isotropy groups get identified and these will be the holonomy groups that allow for the existence of at least two parallel spinors and so on.  The procedure can be carried out to exhaustion. In practice, as the growth of dimension of $\mathrm{Spin}_n$ is polynomial in $n$ while that of $\Delta$ is exponential, this investigation becomes increasingly involved as $n$ increases, especially in the Lorentzian case.  As a result, it has been carried out only in a few cases for small $n$.

A concrete prescription of the Killing spinors is possible by choosing a representative for each trivial module of the holonomy group ${\mathcal H}$ in $\Delta$. In turn, this specifies the embedding of ${\mathcal H}$ in $\mathrm{Spin}_n$. An economical way to achieve this is to choose a realisation of $\Delta$ in terms of forms. This has the advantage that the Killing spinors are explicitly written down.  A similar description has been chosen by the author of  \cite{wang} to give the parallel spinors of Berger manifolds.


\subsection{Killing spinors and geometry}

\subsubsection{Killing spinors}

\begin{table}
\centering
\begin{tabular}{|c|c|c|}
\hline

$L$ & ${\mathcal K}$ & ${\mathrm {parallel~~spinors}}$
 \\
 \hline\hline
$1$ & $Spin(7)$& $1+e_{1234}$ \\
\hline
$2$ &  $SU(4)$&$ 1$
\\
\hline
$3$ & $Sp(2)$&$1,\,\, i(e_{12}+e_{34})$
\\
\hline
$4$ & $Sp(1)\times Sp(1)$&$1,\,\, e_{12}$
\\
\hline
$5$ & $Sp(1)$&$1,\,\, e_{12},\,\, e_{13}+e_{24}$
\\
\hline
$6$ &$U(1)$&$1,\,\, e_{12},\,\, e_{13}$
\\
\hline
$8$ &  $\{1\}$&$ 1,\,\, e_{12},\,\, e_{13},\,\, e_{14}$
\\
\hline
\end{tabular}
\vskip 0.2cm
\caption{\label{tablehola}  {\small
In the  columns are listed the number of  $\hat\nabla$-parallel spinors,   their isotropy groups in $Spin(9,1)$ and their representatives, respectively.
The $\hat\nabla$-parallel spinors are  always real. So if a complex spinor is listed  as a representative,  it is understood that one should
always take its real and imaginary parts.} }
\end{table}

Supersymmetric backgrounds are manifolds ${\mathcal M}$ equipped with a metric $g$ and possibly other fields, which can include forms or connections,  that  admit  Killing spinors, i.e.   solutions to the  Killing spinor equations (KSEs). These equations are the vanishing conditions of the supersymmetry transformations of the fermions of  supergravity theories and always include a parallel transport equation associated with supersymmetry variation of the gravitino.  The classification programme of supersymmetric backgrounds involves the identification of the conditions on the geometry of spacetime ${\mathcal M}$, and  those on the other fields, such that the KSEs admit non-trivial solutions.  Supersymmetric backgrounds have found many applications in theoretical physics,  string theory and differential geometry and include the well known instantons and solitons of gauge theories as well as the Berger manifolds that admit parallel spinors.

In the heterotic theory apart from the spacetime metric $g$ and the 3-form field strength $H$, other fields include the dilaton, $\Phi$, and the curvature, $F$, of a connection with gauge group that we shall not specify here. The dilaton is a real function on the spacetime, which is a 10-dimensional manifold, ${\mathcal M}^{10}$, with Lorentzian signature. The spinor bundle, $S^+$, over ${\mathcal M}^{10}$ is associated with the 16-dimensional positive chirality real irreducible module,  $\Delta_{\bf 16}^+$, of $\mathrm{Spin}(9,1)$.  The KSEs are
\bea
 \hat\nabla_X\epsilon=0~,~~~
 ({\slashed d\Phi}-{1\over12} {\slashed H})\epsilon=0~,~~~
 {\slashed F} \epsilon=0~,
\label{kse}
\eea
where now $\hat\nabla$ is the connection induced on $S^+$ from the tangent bundle connection $\hat\nabla$ in (\ref{hcon}), ${\slashed \omega}$ is the Clifford algebra element associated to the form $\omega$ and $\epsilon$ is a section of $S^+$.

A description of $\Delta_{\bf 16}^+$ in terms of forms can be achieved with the  following vector space construction. Consider the vector space of all forms, $\Lambda^*(\bC^5)$, of  $\bC^5$. To turn this vector space into a Clifford algebra $\mathrm{Cliff}(\bR^{9,1})$ module, choose a Hermitian basis $(\bbe_1, \dots, \bbe_5)$ in $\Lambda^1(\bC^5)$ and observe that
\bea
&&\Gamma_0\zeta=-e_5\wedge \zeta+ e_5\bar\wedge \zeta~,~~~\Gamma_5\zeta=e_5\wedge \zeta+ e_5\bar\wedge \zeta~,
\cr
&&\Gamma_i\zeta=e_i\wedge \zeta+ e_i\bar\wedge \zeta~,~~~\Gamma_{i+5}\zeta=i(e_i\wedge \zeta- e_i\bar\wedge \zeta)~,~~~i=1,2,3,4~,
\label{cl91}
\eea
with $\zeta\in \Lambda^*(\bC^5)$, is a basis in $\mathrm{Cliff}(\bR^{9,1})$, i.e. they satisfy the Clifford algebra relation, where  $\wedge$ is the standard wedge product and $\bar\wedge$ is the inner-derivation operation, i.e.   $e_i\bar\wedge e_j=\delta_{ij}$.  As a $\mathrm{Spin}(9,1)$ module, $\Lambda^*(\bC^5)$ decomposes to two irreducible modules depending on whether the forms are of even or odd degree, $\Lambda^*(\bC^5)=\Lambda^{\mathrm {ev}}(\bC^5)\oplus \Lambda^{\mathrm {od}}(\bC^5)$, and correspond to the positive and negative chirality spinor representations, respectively. Clearly, these representations are complex. However, one can impose a $\mathrm{Spin}(9,1)$-invariant reality condition   via the anti-linear map
$\Gamma_{6789}*$, where $*$ denotes the standard complex conjugation operation.  Using this reality condition, one can construct $\Delta_{\bf 16}^+$ as a
$\mathrm{Spin}(9,1)$-invariant real section of $\Lambda^{\mathrm {ev}}(\bC^5)$.

\begin{theorem}\label{ncom}
The holonomy group of $\hat\nabla$ connection for heterotic backgrounds that admit $\hat\nabla$ spinors can be either non-compact or compact.  In the former case, it is a semi-direct product, ${\mathcal K}\ltimes \bR^8$, with ${\mathcal K}$ given in table \ref{tablehola}, and in latter case, it is one of the groups in table \ref{tableholb}. The number of parallel spinors as well as their representatives for each holonomy group are also given in the tables.
\end{theorem}
\begin{proof}
  The proof of this statement is given in \cite{uggp1, uggp2} and reviewed in \cite{rev}. So no further explanation will be provided here.
\end{proof}

\begin{table}
\centering
\begin{tabular}{|c|c|c|}
\hline

$L$ & ${\mathrm {holonomy~~groups}}$ & ${\mathrm {parallel~~spinors}}$
 \\
 \hline\hline
$2$ & $G_2$&$1+e_{1234},\,\, e_{15}+e_{2345}$
\\
\hline
$4$ & $SU(3)$&$1,\,\, e_{15}$
\\
\hline
$8$ & $SU(2)$&$1,\,\, e_{12}, \,\,e_{15},\,\, e_{25}$
\\
\hline
$16$ & $\{1\}$&$\Delta^+_{{\bf 16}}$
\\
\hline
\end{tabular}
\vskip 0.2cm
\caption{\label{tableholb}  {\small
The description of this table is the same as that tabulated in table \ref{tablehola}. } }
\end{table}

\subsubsection{Geometry}

The geometry of the spacetime ${\mathcal M}^{10}$ depends on the holonomy group ${\mathcal H}$ of $\hat\nabla$  and, in particular, on whether ${\mathcal H}$ is a compact or a non-compact group.  Before we proceed further, we shall assume that {\it all} $\hat\nabla$-parallel spinors are Killing.  This means that all $\hat\nabla$-parallel spinors also solve the remaining two KSEs in (\ref{kse}).  This is not always the case and there are some notable exemptions to this. But such an assumption will simplify the description of the geometry of ${\mathcal M}^{10}$, for the full analysis see \cite{uggp1, uggp2}.

One way to investigate the geometry is to consider the (Killing spinor) form bilinears, $\omega$, which are a generalisation of the Dirac current.  In particular, for any two spinors $\epsilon_1$ and $\epsilon_2$, one can define the $k$-form
\bea
\omega(X_1, \dots, X_k)={1\over k!} \langle \epsilon_1, \sum_{\sigma} (-1)^{|\sigma|} \slashed{X}_{\sigma(1)}\cdots \slashed{X}_{\sigma(k)}\epsilon_2\rangle~,
\eea
where $X_1, \dots, X_k$ are vector fields on ${\mathcal M}^{10}$, $\sigma$ is a permutation of $\{1, \dots, k\}$, $|\sigma|$ is the order of permutation and $\langle \cdot, \cdot\rangle$ is the Dirac $\mathrm{Spin}(9,1)$-invariant inner product.  It is straightforward to demonstrate that all form bilinears $\omega$ are $\hat\nabla$-covariantly constant.

\begin{remark}
One can define the space of fundamental forms  of the KSEs (\ref{kse}), which again we shall denote with  $\mathfrak{f}^{\mathcal H}$, in a similar way to that  $\mathfrak{f}^{\mathcal H}$.
It turns out that all the fundamental forms, ${\mathfrak f}^{\mathcal H}$, of the KSEs (\ref{kse}) in all heterotic geometries are form bilinears.
 Their $\hat\nabla$-covariant constancy condition, together with some additional restrictions on $H$, $\Phi$ and $F$,  imply all the conditions that can be derived by directly solving the KSEs (\ref{kse}) for the spinors tabulated in tables \ref{tablehola} and \ref{tableholb}.
\end{remark}

The properties of fundamental forms for compact and non-compact holonomy groups are different leading to different descriptions of the geometry of spacetime. So the two cases will be separately investigated.

\subsection{Compact holonomy groups}

One characteristic of the geometry of manifolds for which $\hat\nabla$ has compact holonomy group ${\mathcal H}$ is that they admit a number of 1-form bilinears $\kappa$. As $\kappa$ are $\hat\nabla$-parallel, they are no-where vanishing on ${\mathcal M}^{10}$. Moreover, $\hat\nabla \kappa=0$ implies
\bea
{\cal L}_\vkappa g=0~,~~~d\kappa=i_{\vkappa} H~.
\label{kappaeqn}
\eea
Furthermore given two such 1-form bilinears, $\kappa_1$ and $\kappa_2$, then $g(\vkappa_1, \vkappa_2)$ is constant and
\bea
[\vkappa_1, \vkappa_2]=\overrightarrow{i_{\vkappa_1} i_{\vkappa_2} H}~.
\eea

\begin{lemma}\label{lem1}
If  $i_{\vkappa_1} i_{\vkappa_2} dH=0$, the commutator of two 1-form bilinears is $\hat\nabla$-covariantly constant. In particular, this statement holds  for $H$ closed,  $dH=0$.
\end{lemma}
\begin{proof}
This result is a consequence of the Bianchi identity
\bea
\hat R(X, Y; Z ,W)+\mathrm{cyclic} (Y,Z, W)=-\hat\nabla_X H(Y,Z,W)-{1\over2} dH(X,Y,Z,W)~,
\label{bianchi}
\eea
for the curvature, $\hat R$, of $\hat\nabla$ together with $\hat R(\cdot, \cdot; \cdot ,\vkappa)=0$, which arises as integrability condition of the $\hat\nabla$-covariant constancy of $\kappa$, where $X, Y, Z, W$ are vector fields.
\end{proof}

\begin{remark}
The above lemma implies  that either the commutator of two 1-form bilinears closes to another 1-form bilinear or the holonomy group ${\mathcal H}$ of $\hat\nabla$ reduces further to a subgroup of the groups stated in table  \ref{tableholb}.  To avoid such an apparent reduction of the holonomy group, we shall require that the Lie bracket algebra of 1-form bilinears closes on the set.
\end{remark}

\begin{theorem}\label{thm1}
 If the holonomy group ${\mathcal H}$ of $\hat\nabla$ is strictly one of the groups listed in table \ref{tableholb} and $i_{\vkappa_1} i_{\vkappa_2} dH=0$, then  ${\mathcal M}^{10}$ admits the action of one of the Lorentzian Lie algebras tabulated in table \ref{tablelie}.  These Lie algebras are  generated by the 1-form bilinears.
\end{theorem}

\begin{proof}
A direct calculation using the Killing spinors in table \ref{tableholb} reveals that the spacetime admits a timelike  and several spacelike 1-form bilinears whose number depends on the holonomy group.  The restriction of spacetime metric on the space spanned by the 1-form bilinears, which are nowhere vanishing on ${\mathcal M}^{10}$, is the  Lorentzian metric as all their inner products are constant. As the commutator of 1-form bilinears closes on the set, the space spanned by the 1-form bilinears is a Lie algebra equipped with a bi-invariant Lorentzian metric. The latter follows because  the structure constants are given by components of $H$ which is a 3-form. The list of Lorentzian Lie algebras that can occur in each case can be easily identified and they are tabulated in table \ref{tablelie} for each holonomy group ${\mathcal H}$.
\end{proof}

\begin{table}
\begin{center}
\begin{tabular}{|c|c|c|}
\hline
${\mathrm{Holonomy}}$ & ${\mathrm{Dimension}}$&$\mathfrak{Lie}\,{\mathcal G}$
 \\
 \hline
$G_2$&$3$&$\bR^{2,1}~,~\mathfrak{sl}(2,\bR)$
\\
\hline
$SU(3)$ & $4$&$\bR^{3,1}~,~\mathfrak{sl}(2,\bR)\oplus \bR~,~ \mathfrak{su}(2)\oplus \bR~,~\mathfrak{cw}_4$
\\
\hline
$SU(2)$ & $6$&$ \bR^{5,1}~,~\mathfrak{sl}(2,\bR)\oplus\mathfrak{su}(2)~,~\mathfrak{cw}_6$
\\
\hline
\end{tabular}
\end{center}
\vskip 0.2cm
\caption{\label{tablelie}
In the first column, the compact holonomy groups  are stated. In the second column, the number of
1-form  bilinear is given. In the third column, the associated Lorentzian Lie algebras are exhibited.
The structure constants of the 6-dimensional Lorentzian Lie algebras of the $SU(2)$ case
are self-dual. This is a consequence of the dilatino KSE.}
\end{table}

It is well known that if the action of a Lie algebra $\mathfrak{g}$ on a manifold ${\mathcal M}^{10}$ is generated  by  complete vector fields, then it can be integrated to an action of the unique simply connected group ${\mathcal G}$ on ${\mathcal M}^{10}$ that has Lie algebra $\mathfrak{g}$.  As the 1-form bilinears are nowhere vanishing on ${\mathcal M}^{10}$, the action of $\mathfrak{g}$ on ${\mathcal M}^{10}$  has no fixed points.  As a result, the orbits of ${\mathcal G}$ on ${\mathcal M}^{10}$ must be diffeomorphic to ${\mathcal G}/D_p$, where $D_p$ is a discrete subgroup of ${\mathcal G}$ that depends on the orbit ${\cal O}_p$ passing through $p\in {\mathcal M}^{10}$.

\begin{definition}
 The manifold ${\mathcal M}^{10}$ is regular, iff ${\mathcal M}^{10}$ is a principal bundle with the orbits of the action of ${\mathcal G}$ on ${\mathcal M}^{10}$ diffeomorphic to ${\mathcal G}_D={\mathcal G}/D$, where $D$ is a normal discrete subgroup of ${\mathcal G}$. \hfill  	$\vartriangle$
\end{definition}

\begin{theorem}\label{prin}
Suppose that ${\cal L}_{\vkappa} H=0$. All regular manifolds  ${\mathcal M}^{10}$, which are solutions to the KSEs (\ref{kse}), are principal bundles with fibre ${\mathcal G}_D$ and equipped with a principal bundle connection $\lambda\defeq \kappa$  such that
\bea
g=\eta(\lambda, \lambda)+\pi^* \tilde g~,~~H=CS(\lambda)+\pi^* \tilde H~,
\eea
where $\tilde g$ and $\tilde H$ are a metric and a 3-form on the base space $\tilde {\mathcal M}={\mathcal M}^{10}/{\mathcal G}_D$, respectively,  $CS$ is the Chern-Simons form of $\lambda$ and $\eta$ is a Lorentzian metric induced on $\mathfrak{g}$ by restricting the spacetime metric $g$. $\pi$ is the projection map from ${\mathcal M}^{10}$ onto $\tilde {\mathcal M}$.
\end{theorem}
\begin{proof}
The proof follows from arguments presented in \cite{uggp1, uggp2} which will not be repeated here.
\end{proof}

Apart from the 1-form bilinears, ${\mathcal M}^{10}$ admits additional form bilinears $\phi$.  All these satisfy $i_\vkappa\phi=0$ and so there are transverse to the  subspace of $T{\mathcal M}^{10}$ spanned by the 1-form bilinears. For regular spacetimes, these forms can be used to further restrict the geometry of the orbit space $\tilde {\mathcal M}$ and put some restrictions on the principal bundle connection $\lambda$. The space $\tilde {\mathcal M}$ is  equipped with a $\tilde g$-metric connection, $\hat{\tilde \nabla}$, with torsion $\tilde H$. A very brief description of the geometry of  $\tilde {\mathcal M}$ and the restrictions on $\lambda$ for each holonomy group stated in table \ref{tableholb} is as follows, see \cite{uggp1, uggp2} for more details.

$G_2$:   The holonomy group of $\hat{\tilde \nabla}$ is included in $G_2$, $\tilde {\mathcal M}^7$ is conformally balanced\footnote{Given a (fundamental) form $\phi$, define the 1-form $\theta_\phi= c(\phi) *(\phi\wedge *d\phi)$, where $c$ is a normalisation constant and $*$ denotes the Hodge duality operation. A manifold is conformally balanced with respect to $\phi$, iff $\theta_\phi=2 d\Phi$, where for heterotic geometries $\Phi$ is the dilaton.}  with respect to the fundamental $G_2$ 3-form $\tilde\varphi$ and the curvature of $\lambda$ is a $G_2$-instanton on $\tilde {\mathcal M}^7$, for more details on $G_2$ structures see \cite{gray, stefang2, chiossi}. The fundamental form of the $G_2$ structure on ${\mathcal M}^{10}$ is the pull-back of $\tilde\varphi$. The 3-form $\tilde H$ can be expressed in terms of $\tilde\varphi$ and its first order derivatives, and the metric $\tilde g$. If $\mathfrak{g}$ is abelian, $\tilde H$ is orthogonal to $\tilde\varphi$.

$SU(2)$: The holonomy group of $\hat{\tilde \nabla}$ is included in $SU(2)$, $\tilde {\mathcal M}^4$ is a conformally balanced hyper-K\"ahler with torsion (HKT) manifold \cite{hkt, poon} and so conformal to a hyper-K\"ahler one -- the Lee forms associated to the three Hermitian forms $\tilde\omega_r$ for HKT manifolds are equal.  The curvature of $\lambda$ is an anti-self dual instanton on $\tilde {\mathcal M}^4$. $\tilde H$ can be expressed in terms of the dilaton $\Phi$ \cite{callan}.

$SU(3)$:  If $\mathfrak{g}=\bR^{3,1}$, then  the holonomy group of $\hat{\tilde \nabla}$ is contained in $SU(3)$ and  $\tilde {\mathcal M}^6$ is a complex conformally balanced manifold with respect to the Hermitian  2-form $\tilde \omega$. Furthermore, the curvature of the connection $\lambda$ is an $SU(3)$ instanton,i.e $\lambda$ is a Hermitian-Einstein connection with vanishing ``cosmological'' constant; for more details on $SU(3)$ structures see \cite{chiossi, goldstein, poonb, sunspn}. However if $\mathfrak{g}$ is non-abelian, then the holonomy group of $\hat{\tilde \nabla}$ is contained in $U(3)$, $\tilde {\mathcal M}^6$ is a complex conformally balanced K\"ahler with torsion (KT) manifold \cite{hkt}  with respect to $\tilde \omega$ and the curvature of $\lambda$ is a $U(3)$-instanton, i.e $\lambda$ is again a Hermitian-Einstein connection but with non-vanishing cosmological constant.  In both cases, $\tilde H$ can be expressed in terms of $\tilde \omega$ and its first derivative,  and $\tilde g$.

\begin{remark}
If ${\mathcal M}^{10}$ is not regular, it is still possible to write the metric $g$ and the 3-form $H$ on ${\mathcal M}^{10}$ as in theorem \ref{prin}, especially on the subset of ${\mathcal M}^{10}$ that it is the union of principal orbits of ${\mathcal G}$, as the subgroups $D_p$ are discrete. Also, the comments above on the geometry of $\tilde {\mathcal M}$ will still locally apply.
\end{remark}

\subsection{Non-compact holonomy groups}\label{sec:nullhetgeom}

If the holonomy group of $\hat\nabla$ is non-compact, then ${\mathcal M}^{10}$ admits a single null $\hat\nabla$-covariantly constant 1-form bilinear $\kappa$, $g(\vkappa, \vkappa)=0$. This implies that $\kappa$ satisfies the equations (\ref{kappaeqn}). Moreover, all other form bilinears, $\phi$, are null along $\kappa$, i.e. they satisfy (\ref{nullcon}) and so $\phi\in \Omega^*_{\kappa, \hat\nabla}$.
To describe the geometry of ${\mathcal M}^{10}$, a decomposition is needed of the various tensors involved into directions along $\vkappa$ and into directions transverse to $\vkappa$. As $\vkappa$ is null, one cannot use an orthogonal decomposition and as a result there is an inherent ambiguity in the definition of transverse directions as it has been explained in section \ref{sec:null}.

To describe some of the additional conditions on the geometry of ${\mathcal M}^{10}$ implied by the KSEs (\ref{kse}), introduce a local pseudo-orthonormal frame $(\bbe^-, \bbe^+, \bbe^i; i=1,\dots 8)$ with $\bbe^-=\kappa$ as in section \ref{sec:null} and write the metric as  $g=2 \bbe^-\bbe^++\delta_{ij}\bbe^i\bbe^j$, where the open set labelling has been suppressed. If the holonomy group of $\hat\nabla$ reduces to ${\mathcal K}\ltimes\bR^8$, then the patching condition  for the
pseudo-orthonormal frame is as in (\ref{patch}) with $(O_{\alpha\beta}, w_{\alpha\beta})\in {\mathcal K}\ltimes\bR^8$.

Next, observe that $\kappa\wedge i_\vkappa H$ transforms as a 2-form under the ${\mathcal K}$ transformations of  the patching conditions (\ref{patch}) now associated  with the structure group ${\mathcal K}\ltimes\bR^8$.  As a result, it can be decomposed pointwise as $\Lambda^2\bR^8=\mathfrak{k}\oplus \mathfrak{k}^\perp$ everywhere on ${\mathcal M}^{10}$, where $\mathfrak{k}$ is the Lie algebra of ${\mathcal K}$ and $\mathfrak{k}^\perp$ is its orthogonal complement. This follows  from the decomposition of $\mathfrak{so}(8)=\Lambda^2\bR^8$ and $\mathfrak{k}\subset \mathfrak{so}(8)$. The KSEs (\ref{kse})  imply that
\bea
\kappa\wedge i_\vkappa H\vert_{\mathfrak{k}^\perp }=0~,
\eea
in all cases.
One consequence of this is that ${\cal L}_\vkappa\phi=0$ for all form bilinears $\phi$.

\begin{theorem}
Let the holonomy group of $\hat\nabla$ be the non-compact group, ${\mathcal K}\ltimes \bR^8$, where ${\mathcal K}$ is  one of the groups tabulated  in table \ref{tablehola}.
All form bilinears $\phi$ are null forms along $\kappa$, i.e.  $\phi\in \Omega^*_{\kappa, \hat\nabla}$, and ${\cal L}_\vkappa\phi=0$.    The metric $g$ and 3-form $H$ on ${\mathcal M}^{10}$ that solve the KSEs (\ref{kse}) can be expressed as
\bea
g=2 \bbe^-\bbe^++\delta_{ij} \bbe^i \bbe^j~,~~~H=\bbe^-\wedge d \bbe^-+ H^T~,
\eea
such that $i_\vkappa H^T=0$. Moreover, $\bbe^-\wedge H^T$ is determined by the  form bilinears and their first derivatives, and the metric $g$.
\end{theorem}
\begin{proof}
We shall not present a proof here.  It is rather lengthy to describe all conditions on the geometry of ${\mathcal M}^{10}$ as well as explicit expressions for $H^T$ in each case. These can be found in \cite{uggp1, uggp2} and they have been reviewed in \cite{rev}.
\end{proof}

\section{Heterotic inspired geometries}\label{sec:four}

\subsection{Examples with Euclidean signature}

Although the KSEs (\ref{kse}) naturally arise on 10-dimensional manifolds with Lorentzian signature, they can also be considered on any n-dimensional manifold of any signature. Here, we shall consider  8-dimensional manifolds equipped with a metric $g$ of Euclidean signature, a 3-form $H$ and a scalar field $\Phi$ the dilaton.  The task is to find such solutions to the KSEs (\ref{kse}), where  now the spinor bundle is associated with the representation $\Delta_{\bf 16}=\Delta^+_{\bf 8}\oplus \Delta^-_{\bf 8}$ of ${\mathrm{Spin}}(8)$ --  $\Delta^\pm_{\bf 8}$ are the irreducible real chiral and anti-chiral representations of ${\mathrm{Spin}}(8)$.  This choice of spinor representation arises because the real  $\Delta^+_{\bf 16}$ representation of ${\mathrm{Spin}}(9,1)$ decomposes under the subgroup ${\mathrm{Spin}}(8)$ as that of the $\Delta_{\bf 16}$ above.
Again, we shall assume that {\it all} $\hat\nabla$-covariantly constant spinors are Killing spinors, i.e.{\it  all} solutions of the gravitino KSE also solve the other two KSEs
(\ref{kse}) of the heterotic theory.

To describe a realisation of $\Delta_{\bf 16}$ in terms of forms, consider  $\Lambda^*(\bC^4)$, and  turn this vector space into a  $\mathrm{Cliff}(\bR^{8})$ module as
\bea
&&\Gamma_i\zeta=e_i\wedge \zeta+e_i\bar\wedge \zeta~,~~~\Gamma_{i+4}\zeta=i(e_i\wedge \zeta-e_i\bar\wedge \zeta)~,~~~i=1,2,3,4~,
\eea
with $\zeta\in \Lambda^*(\bC^4)$ and $(e_i; i=1,\dots, 4)$ a Hermitian basis in $\Lambda^1(\bC^4)$. The rest of the details including the definition of the various operations are similar to those described below (\ref{cl91}).  As a $\mathrm{Spin}(8)$ module, $\Lambda^*(\bC^4)$ decomposes into two irreducible modules depending on whether the forms are of even or odd degree, $\Lambda^*(\bC^4)=\Lambda^{\mathrm {ev}}(\bC^4)\oplus \Lambda^{\mathrm {od}}(\bC^4)$,  and correspond to the positive and negative chirality representations,  respectively. Clearly, these representations are complex. However, one can impose a $\mathrm{Spin}(8)$-invariant reality condition   via the anti-linear map
$\Gamma_{6789}*$, where $*$ denotes the standard complex conjugation operation.  Using this reality condition, one can construct $\Delta_{\bf 8}^+$ ($\Delta_{\bf 8}^-$) as a
$\mathrm{Spin}(8)$-invariant real section of $\Lambda^{\mathrm {ev}}(\bC^4)$ ($\Lambda^{\mathrm {od}}(\bC^4)$).

\begin{theorem}\label{thm3}
The holonomy groups ${\mathcal H}$ of the $\hat\nabla$ connection on ${\mathcal M}^8$ that admit $\hat\nabla$-parallel spinors are tabulated in  tables \ref{tablehola} with ${\mathcal H}={\mathcal K}$ and \ref{tableholc}.  The number of parallel spinors as well as their representatives for each case are also given in same tables.
\end{theorem}
\begin{proof}
  The proof of this statement is similar to that given in \cite{uggp1, uggp2} for heterotic geometries on 10-dimensional Lorentzian signature manifolds and reviewed in the previous sections.
  The only difference is that  the parallel spinors in table \ref{tableholc} have been represented with both even and odd degree forms
  while those in the Lorentzian case given in table \ref{tableholb} are represented with even degree forms. Despite this difference, the methodology and many of the details of the proof remain the same.
\end{proof}

\begin{table}
\centering
\begin{tabular}{|c|c|c|}
\hline

$L$ & ${\mathrm {holonomy~~groups}}$ & ${\mathrm {parallel~~spinors}}$
 \\
 \hline\hline
$2$ & $G_2$&$1+e_{1234},\,\, e_{1}+e_{234}$
\\
\hline
$4$ & $SU(3)$&$1,\,\, e_{1}$
\\
\hline
$8$ & $SU(2)$&$1,\,\, e_{12}, \,\,e_{1},\,\, e_{2}$
\\
\hline
\end{tabular}
\vskip 0.2cm
\caption{\label{tableholc}  {\small
The description of this table is the same as that tabulated in table \ref{tablehola}. } }
\end{table}

To investigate the geometry of manifolds admitting Killing spinors, let us begin with those whose holonomy group is tabulated in table \ref{tableholc}. As in the Lorentzian case, we consider the form bilinears of Killing spinors.  It turns out that the Killing spinors of table \ref{tableholc} give rise to  $\hat\nabla$-parallel 1-form bilinears.  In particular,  lemma \ref{lem1} still applies and the commutator of two 1-form bilinears is $\hat\nabla$-parallel.  As a result theorem \ref{thm1} can be adapted in this case.  In particular, one has the following.

\begin{theorem}\label{thm2}
If the holonomy group of $\hat\nabla$ is strictly one of the groups listed in table \ref{tableholc} and $i_{\vkappa_1} i_{\vkappa_2} dH=0$ for any two 1-form bilinears $\kappa_1$ and $\kappa_2$, then  ${\mathcal M}^{8}$ admits the action of a Euclidean signature Lie algebra generated by the 1-form bilinears. The associated  Lie algebras, $\mathfrak{g}$,  are tabulated in table \ref{tablelieb}.
\end{theorem}
\begin{proof}
It can be verified by a direct computation using the Killing spinors tabulated in table \ref{tableholc} that the number of 1-form bilinears are 1, 2 and 4, respectively. From the assumptions of the theorem, no apparent reduction of the holonomy group is allowed to a subgroup of those listed in table \ref{tableholc}. This  implies that the set 1-form bilinears closes under Lie brackets. As the structure constants of Lie algebra are skew-symmetric with respect to the induced metric, one concludes that in the first two cases the Lie algebra is abelian. In the $SU(2)$ holonomy case more possibilities could have risen -- they are allowed by the gravitino KSE, e.g. ${\mathfrak g}=\bR\oplus \mathfrak{so}(3)$ -- but the dilatino KSE implies that the Lie algebra is abelian.
\end{proof}

\begin{table}
\begin{center}
\begin{tabular}{|c|c|c|}
\hline
${\mathcal H}$ & ${\mathrm{Dimension}}$&${\mathfrak g}$
 \\
 \hline
$G_2$&$1$&$\bR$
\\
\hline
$SU(3)$ & $2$&$\bR^2$
\\
\hline
$SU(2)$ & $4$&$ \bR^4$
\\
\hline
\end{tabular}
\end{center}
\vskip 0.2cm
\caption{\label{tablelieb}
In the first column, holonomy groups, ${\mathcal H}$,  are stated. In the second column, the number of
1-form  bilinears, $\kappa$, is given. In the third column, the associated  Lie algebras are exhibited.
}
\end{table}

The simply connected groups ${\mathcal G}$ associated with the Lie algebras $\mathfrak{g}$ tabulated in table \ref{tablelieb} are $\bR$, $\bR^2$ and $\bR^4$, respectively.  If the vector fields generated by the action of $\mathfrak{g}$ on ${\mathcal M}^8$ are complete, then the orbits of the action ${\mathcal G}$ on ${\mathcal M}^8$  will be ${\mathcal G}/D_p$, $p\in {\mathcal M}^8$, as the vector fields are nowhere vanishing, where $D_p$ is a discrete subgroup of ${\mathcal G}$.  The geometry of ${\mathcal M}^8$ in the regular case, where all orbits of ${\mathcal G}$ on ${\mathcal M}^8$ have the same isotropy group $D$, is as follows.

\begin{theorem}\label{prin2}
Suppose that ${\cal L}_\vkappa H=0$ for all 1-form bilinears $\kappa$. All regular solutions ${\mathcal M}^8$ to the KSEs (\ref{kse}) are principal bundles with fibre ${\mathcal G}_D={\mathcal G}/D$ and equipped with a principal bundle connection $\lambda\defeq \kappa$  such that
\bea
g=\eta(\lambda, \lambda)+\pi^* \tilde g~,~~H=CS(\lambda)+\pi^* \tilde H~,
\eea
where $\tilde g$ and $\tilde H$ are a metric and a 3-form on the base space $\tilde{\mathcal M}={\mathcal M}^{8}/{\mathcal G}_D$, respectively,  $CS$ is the Chern-Simons form of $\lambda$ and $\eta$ is the  restriction the  metric $g$ of ${\mathcal M}^8$ on  $\mathfrak{g}$.  Moreover, the connection $\hat{\tilde \nabla}$ on $\tilde{\mathcal M}$ has holonomy group $G_2$, $SU(3)$ and $SU(2)$, respectively.  In addition, the curvature of $\lambda$ on  $\tilde{\mathcal M}$ is a $G_2$, $SU(3)$, i.e. $\lambda$ is Hermitian-Einstein with zero cosmological constant,  and anti-self-dual instanton, respectively.
\end{theorem}
\begin{proof}
The proof follows using similar arguments to those presented in \cite{uggp1, uggp2} that will not repeated here.
\end{proof}

In all cases, ${\mathcal M}^8$ admits additional form bilinears $\phi$. In addition, it turns out that  $i_\vkappa\phi=0$ and ${\cal L}_\vkappa\phi=0$.   For regular spacetimes, $\phi$ is the pull back of a form $\tilde \phi$ on the orbit space $\tilde{\mathcal M}$.  The forms $\tilde\phi$ can be used to further restrict the geometry of $\tilde{\mathcal M}$ as well as the curvature of the principal bundle connection $\lambda$.  A brief summary of the geometry of $\tilde{\mathcal M}$ in each case is as follows.

$G_2$: The holonomy group of the connection $\hat{\tilde\nabla}$ is included in $G_2$.  $\tilde{\mathcal M}^7$ is conformally balanced with respect to the fundamental $G_2$ 3-form $\tilde\varphi$. $\tilde H$ is completely determined in terms of $\tilde \varphi$, its first derivatives and the metric $\tilde g$. $\tilde H$ is orthogonal  to $\tilde \varphi$, and the curvature of $\lambda$ is a $G_2$ instanton on $\tilde{\mathcal M}^7$.

 $SU(3)$:  The holonomy group of the connection $\hat{\tilde\nabla}$ is included in $SU(3)$. $\tilde{\mathcal M}^6$ is a complex manifold and conformally balanced with respect to the Hermitian form $\tilde\omega$. $\tilde H$ is completely determined in terms of $\tilde \omega$, its first derivatives and $\tilde g$, and the curvature of $\lambda$ is a $SU(3)$ (Hermitian-Einstein) instanton on $\tilde{\mathcal M}^6$.

 $SU(2)$: The holonomy group of the connection $\hat{\tilde\nabla}$ is included in $SU(2)$. $\tilde{\mathcal M}^4$ is a conformally balanced HKT manifold -- for HKT manifolds the Lee forms associated to the three Hermitian forms, $\tilde\omega_r$, $r=1,2,3$, are equal.    Therefore, $\tilde{\mathcal M}^4$ is conformal to a hyper-K\"ahler manifold. $\tilde H$ is completely determined in terms of $\tilde \omega_r$, its first derivatives and $\tilde g$  or equivalently in terms of $\tilde g$ and the first derivatives of the dilaton $\Phi$.   The curvature of $\lambda$ is a an anti-self dual instanton on $\tilde{\mathcal M}^4$.

 \begin{remark}
 The possibility of a trivial holonomy group has not been included in table \ref{tableholc}. This is because we have assumed that all $\hat\nabla$-covariantly constant spinors also solve the remaining two KSEs in (\ref{kse}). Clearly, group manifolds with $\hat\nabla$ the left invariant connection are solutions to the gravitino KSE and $\hat\nabla$ has trivial holonomy.  However, not all $\hat\nabla$-parallel spinors solve the dilatino KSE, which violates one of our assumptions. If one insists that all $\hat\nabla$-parallel spinors solve the remaining two KSEs as well, then ${\mathcal M}^8$ is locally isometric to $\bR^8$ with $H=0$ and constant dilaton.
 \end{remark}

Next, let us consider the geometry of ${\mathcal M}^8$ admitting a $\hat\nabla$ connection with holonomy one of the groups listed in table \ref{tablehola} that solve the KSEs (\ref{kse}).  All the form bilinears have even degree. So  such manifolds  do not admit 1-form bilinears. A description of their geometry is as follows.

\begin{theorem}\label{prin3}
The holonomy group ${\mathcal H}$ of $\hat\nabla$ on ${\mathcal M}^8$ is  contained in one of those listed in table \ref{tablehola}. In addition, if the holonomy of $\hat\nabla$ is contained in $SU(4)$, $Sp(2)$, $\times^2 Sp(1)$, $Sp(1)$ or $U(1)$, then ${\mathcal M}^8$ must be a complex manifold   with respect to all compatible complex structures. Furthermore,  ${\mathcal M}^8$ is conformally balanced with respect to the fundamental form of $\mathrm{Spin}(7)$ in the holonomy $\mathrm{Spin}(7)$ case and with respect to the compatible Hermitian forms in the remaining cases. $H$ is completely determined by the  form bilinears and the metric $g$ of ${\mathcal M}^8$.
\end{theorem}
\begin{proof}
The results for the $\mathrm{Spin}(7)$ and $SU(4)$ follow from those of \cite{uggp1}, see also \cite{stefanspin7}. The solutions with holonomy $Sp(2)$ are  conformally balanced HKT manifolds. In the $\times^2 Sp(1)$, $Sp(1)$ and $U(1)$ cases, $\mathfrak{f}^{\mathcal{H}}$ is spanned by 2-form bilinears, $\omega$, which are Hermitian forms. ${\mathcal M}^8$ is a complex manifold with respect to all compatible complex structures $\vec\omega$.  Moreover, $\vec\omega$ span a basis of the Clifford algebras ${\mathrm {Cliff}} (\bR^k)$  with quadratic form whose signature is $-k$ for $k=3,4, 5$, respectively.    For a further description of these geometries, see \cite{uggp2}.
\end{proof}

\begin{remark}
It is a consequence of a theorem in \cite{sigp1} that there are no compact, complex, conformally balanced manifolds, ${\mathcal M}^8$, with closed 3-form, $dH=0$, whose holonomy group of the $\hat\nabla$-connection is included in $SU(n)$.  This theorem clearly extends to manifolds whose holonomy  group is included in $Sp(2)$, $\times^2 Sp(1)$, $Sp(1)$ and $U(1)$ under the same assumptions.  This is because  all these are special cases of $SU(n)$ for $n=4$.  To my knowledge, there are no examples of manifolds, compact or non-compact, for which the holonomy group of the $\hat\nabla$ connection is strictly either $Sp(1)$ or $U(1)$.

\end{remark}

\subsection{Examples with Lorentzian signature}\label{lorexn}

Motivated by the solution of the KSEs (\ref{kse}) with non-compact holonomy group, one can begin from a compact group ${\mathcal K}$ and a representation of ${\mathcal K}$ on $\bR^p$ and construct manifolds, ${\mathcal M}^{p+2}$ with holonomy ${\mathcal K}\ltimes\bR^p$.  For simplicity, one can take ${\mathcal K}$ to be one of the Berger groups $SO(n)$, $U(n)$, $SU(n)$, $Sp(n)$, $Sp(n)\cdot Sp(1)$ and as a representation the fundamental vector representation of these groups. Then, one can construct Lorentzian signature geometries  by demanding the existence of a metric connection, $\nabla^{\mathcal H}$,  with holonomy ${\mathcal H}$ one of the groups  $SO(n)\ltimes\bR^n\,(n+2)$, $U(n)\ltimes\bR^{2n}\,(2n+2)$, $SU(n)\ltimes\bR^{2n}\,(2n+2)$, $Sp(n)\ltimes\bR^{4n}\,(4n+2)$
$Sp(n)\cdot Sp(1)\ltimes\bR^{4n}\,(4n+2)$,  where in parenthesis is the dimension of the associated Lorentzian signature manifold ${\mathcal M}$. One way to restrict the holonomy group of $\nabla^{\mathcal H}$ in this way is to demand that the spacetime admits  a $\nabla^{\mathcal H}$-covariantly constant null 1-form $\kappa$ as well as additional $\nabla^{\mathcal H}$-covariantly constant forms
$\phi$ thar satisfy the properties stated in (\ref{nullcon}).  In addition, one requires that locally $\phi$ can  be written as in (\ref{exL}) with $\breve \phi$ at every patch represented by the fundamental forms of the compact subgroup of the holonomy group.  Furthermore, the spacetime metric in a compatible frame to the
${\mathcal K}\ltimes\bR^p$ structure is written as $g =2 \bbe^-\bbe^++\delta_{ij} \bbe^i \bbe^j$ with $\bbe^-\defeq \kappa$.  This construction can be also extended to the $G_2$ case but ${\mathcal M}$ will be 9-dimensional.

For  geometries directly inspired by the heterotic theory, $\nabla^{\mathcal H}$ is identified with the connection with skew-symmetric torsion\footnote{For the properties of manifolds with a $Sp(n)\cdot Sp(1)$-structure, see e.g. \cite{salamon, QKT1, QKT2, QKT3, QKT4}.} $\hat\nabla$.  Though observe  that if the holonomy group of $\hat\nabla$ is strictly $SO(n)\ltimes\bR^n$, $U(n)\ltimes\bR^{2n}$ and $Sp(n)\cdot Sp(1)\ltimes\bR^{4n}$, the spacetime ${\mathcal M}$ will not be a solution of the KSEs (\ref{kse}) -- nevertheless, such manifolds can still be solutions to the field equations of the heterotic string.

\section{Algebraic structures on fundamental forms}

\subsection{Sigma model holonomy symmetries}\label{sec:holsym}

It has been known for sometime that $\hat\nabla$-covariantly constant forms on a spacetime ${\mathcal M}$ generate  infinitesimal symmetries, referred to as holonomy symmetries,  in some sigma model actions.  As infinitesimal symmetries of actions are naturally endowed with a commutator that always closes to a symmetry, there is the possibility that this induces  an underlying algebraic structure on the space of $\hat\nabla$-covariantly constant forms.  Before we proceed to confirm that this is the case for a class of such symmetries, let us first review some of the properties of the holonomy symmetries.

Typically, the properties of the holonomy symmetries are investigated in the context of a 2-dimensional (string) supersymmetric sigma models with target space the manifold ${\mathcal M}$. In particular, it has been noticed in \cite{phgpw2, phgpw1} that the commutator of these symmetries is that of a W-algebra.  In turn this indicates that strings propagating on such backgrounds exhibit a W-algebra of symmetries.  This is a larger structure than the expected (appropriate) worldsheet supersymmetry of the theory. However, a similar analysis can be carried out for the holonomy symmetries of  1-dimensional (particle) $N=1$ supersymmetric sigma models with target space again ${\mathcal M}$. As the holonomy symmetries of 2-dimensional sigma models have already been explored, here we shall describe those of 1-dimensional sigma models.
 The classical fields of such a  sigma model  are  maps, $X$, from the worldline superspace $\Xi^{1|1}$, which is a (flat) supermanifold with one Grassmannian even and one Grassmannian odd coordinates, $(\tau\,\vert\, \theta)$, into the spacetime ${\mathcal M}$, $X:~\Xi^{1|1}\rightarrow {\mathcal M}$.  An action  for these fields \cite{coles} is
\bea
S=-i \int d\tau d\theta\,  \Big((X^*g)_{\mu\nu}DX^\mu \partial_\tau X^\nu -{i\over6}(X^*H)_{\mu\nu\rho} DX^\mu DX^\nu DX^\rho\Big)  ~,
\label{act1}
\eea
where $g$ is a spacetime metric, $H$ is a 3-form on ${\mathcal M}$, $D$ is the superspace derivative with  $D^2=i\partial_\tau$, and $X^*g$ and $X^*H$ denote the pull back of $g$ and $H$ on $\Xi^{1|1}$, respectively. Naturally, $H$ is identified with the heterotic 3-form field strength.  We have also written the action in a coordinate basis for clarity.

Let $L$  be a vector $\ell$-form, $L\in \vec{\Omega}^\ell({\mathcal M})$,   on the sigma model target space ${\mathcal M}$ and consider the infinitesimal transformation
\bea
\delta_L X^\mu=a_L L^\mu{}_L DX^L \defeq a_L L^\mu{}_{\lambda_1\dots \lambda_\ell} DX^{\lambda_1}\dots DX^{\lambda_\ell}~,~~~
\label{Gsym}
\eea
where $a_L$ is a parameter chosen such that $\delta_L X$ is Grassmannian even, the index $L$ is the multi-index $L=\lambda_1 \dots \lambda_\ell$ and $DX^L=DX^{\lambda_1}\cdots DX^{\lambda_\ell}$.  Such transformations \cite{phgpw1, phgpw2} leave the action  (\ref{act1}) invariant provided\footnote{The invariance of the action (\ref{act1}) under the transformations  (\ref{Gsym}) can be achieved with weaker conditions than those stated in (\ref{invcon}). However, the conditions (\ref{invcon}) will suffice for the purpose of this work.}
\bea
\hat\nabla L=0~, ~~i_L dH=0~~ \text{and}~~ L=\vec\phi~,  ~~\phi\in \Omega^{\ell+1}({\mathcal M})~.
\label{invcon}
\eea
 Moreover, the parameter $a_L$ satisfies $\partial_\tau a_L=0$, i.e. $a_L=a_L(\theta)$. It is straightforward to observe that for $\ell=0$,  $\hat\nabla L=0$ implies that $L$ is a Killing vector field and  $i_\vkappa H=d\kappa$, where $L=\vkappa$. This together with $i_L dH=0$ gives ${\cal L}_\vkappa H=0$. Similarly for $L=\vec\phi$ with $\ell>0$, $\hat\nabla \phi=0$ implies that $d\phi=i_{\vec \phi} H$ and together with $i_{\vec \phi} dH=0$, one has that $d_{\vec \phi} H=0$.  Note that the conditions
 in (\ref{invcon}), apart from $i_LdH=0$, are those satisfied by the fundamental forms of heterotic  (inspired) geometries. It turns out that $i_LdH=0$ is also satisfied by heterotic geometries either because $dH=0$ or because the correction to $dH$ due to the anomaly is appropriately restricted due to consistency conditions. Whether the condition $i_LdH=0$ is imposed on the heterotic (inspired) geometries or not, it does not affect the computation of the commutator (\ref{comm}) of the holonomy symmetries that we present below. However, for the fundamental forms of heterotic (inspired) geometries to generate holonomy symmetries in sigma model actions, we have to assume that their fundamental forms  satisfy $i_L dH=0$.

The commutator of two transformations (\ref{Gsym}) on the field $X$  is similar to that of 2-dimensional sigma models that has been explored in detail in  \cite{gpph}. Because of this, we shall only summarise some of the key formulae.  The commutator of two  transformations (\ref{Gsym}) on the field $X$ generated by the vector  $\ell$-form $L$ and the vector  $m$-form $M$  can be written as
\bea
[\delta_L, \delta_M]X^\mu= \delta_{LM}^{(1)} X^\mu+\delta_{LM}^{(2)} X^\mu+\delta_{LM}^{(3)} X^\mu~,
\label{comm}
\eea
with
\bea
\delta_{LM}^{(1)} X^\mu=a_M a_L N(L,M)^\mu{}_{L M} DX^{L M}~,
\eea
\bea
(\delta_{LM}^{(2)} X)_\mu&=&\big (-m a_M Da_L (L\cdot M)_{\nu L_2, \mu  M_2}
\cr
&&
\qquad\qquad
+ \ell (-1)^{(\ell+1) (m+1)} a_L D a_M (L\cdot M)_{\mu L_2,\nu  M_2}\big)
 DX^{\nu L_2 M_2}~,
\eea
and
\bea
(\delta_{LM}^{(3)} X)_\mu=-i \ell m (-1)^\ell a_M a_L \Big((L\cdot M)_{\mu L_2, \nu  M_2}+(L\cdot M)_{\nu L_2, \mu  M_2}\Big) \partial_{\pp}X^\nu DX^{L_2 M_2}~,
\eea
where
\bea
(L\cdot M)_{\mu L_2,\nu  M_2} dx^{L_2 M_2}\defeq L_{\rho \mu L_2} M^\rho{}_{\nu  M_2} dx^{L_2 M_2}~,
\eea
and
\bea
N(L,M)^\mu{}_{L M} \partial_\mu\otimes dx^{L M}&=&\Big (L^\nu{}_L\partial_\nu M^\mu{}_{ M}- M^\nu{}_ M \partial_\nu L^\mu{}_L-\ell L^\mu{}_{\nu L_2} \partial_{\lambda_1} M^\nu{}_ M
\cr
&&
+m M^\mu{}_{\nu  M_2}
\partial_{\mu_1} L^\nu{}_L\Big) \partial_\mu\otimes dx^{L M}~,
\eea
is the Nijenhuis tensor of $L$ and $M$.
The multi-indices $L$ and $ M$ stand for $L=\lambda_1\dots\lambda_\ell$ and $ M=\mu_1\dots \mu_m$ while the multi-indices $L_2$ and $ M_2$ stand for $L_2=\lambda_2\dots\lambda_\ell$ and $ M_2=\mu_2\dots \mu_m$, respectively.
 Furthermore, after using that $L$ and $M$ are $\hat\nabla$-covariantly constant, the Nijenhuis tensor can be rewritten as
\bea
&&N_{\mu L M} dx^\mu\otimes dx^{L M}=  - (\ell+m+1)H_{[\mu|\nu\rho|} L^\nu{}_L M^\rho{}_{ M]} dx^{\mu L M}
\cr
&& \qquad\qquad +{1\over2}\ell m \Big(H^{\rho}{}_{\lambda_1\mu_1} ( L\cdot M)_{\mu|L_2|, \rho M_2}+ \mu\leftrightarrow\rho\Big) \, dx^\mu\otimes dx^{L M} ~.
\eea
This concludes the description of the commutator of two holonomy symmetries. The Noether conserved current of a symmetry generated by the $(\ell+1)$-form $\phi$, $\vec\phi=L$,  is
\bea
J_L=\phi_{\lambda_1\dots\lambda_{\ell+1}} DX^{\lambda_1\dots\lambda_{\ell+1}}~,~~~
\label{curL}
\eea
It can be easily seen that $\partial_\tau J_L=0$ subject to field  equations  of (\ref{act1}).

\begin{remark}
As both transformations $\delta_L$ and $\delta_M$ are symmetries of (\ref{act1}), the right hand side of their commutator, $[\delta_L, \delta_M]$,  is also a symmetry. This guarantees the closure of the algebra of symmetries -- though a refinement of this will be presented below.  The appearance of the Nijenhuis tensor of $L$ and $M$ in the right hand side of (\ref{comm}) indicates that there is a relation between the commutator of holonomy symmetries and the commutator, $[d_L, d_M]$, of exterior derivations $d_L$ and $d_M$ associated to $L$ and $M$. However, closure in the latter case is not guaranteed for  manifolds with a reduced structure group.  This is because if $L$ and $M$ are constructed from the fundamental forms, $\mathfrak{f}^{\mathcal H}$, of an ${\mathcal H}$-structure, and so they are $\nabla^{\mathcal H}$-covariantly constant, the Nijenhuis tensor, $N(L,M)$, may not be $\nabla^{\mathcal H}$-covariantly constant. One way to resolve this is to explore the Gray-Hervella classes of the ${\mathcal H}$-structure and specify those for which $N(L,M)$ is $\nabla^{\mathcal H}$-covariantly constant.  Alternatively for $\nabla^{\mathcal H}=\hat\nabla$, one can explore the analogy between the commutator of symmetry variations with that of derivations and use it to define a bracket on $\mathfrak{f}^{\mathcal H}$ such that it closes as an algebra in $\Omega^*_{\nabla^{\mathcal H}}$.

\end{remark}

One of the issues that arises in the investigation of the commutator of symmetries (\ref{comm}) is that the individual variations $\delta_{LM}^{(1)}$, $\delta_{LM}^{(2)}$ and $\delta_{LM}^{(3)}$ may not be symmetries of the action (\ref{act1}) -- although of course, their sum is. To rectify this, consider symmetries generated by  the vector (q+1)-form
\bea
S={1\over (q+1)!}S^\mu{}_{\nu\Xi}\,\partial_\mu\otimes dx^{\nu\Xi}={1\over (q+1)!} g^{\mu \lambda} S_{\lambda, \nu\Xi}\,\partial_\mu\otimes dx^{\nu\Xi}\defeq {1\over (q+1)!} \delta^\mu{}_{\nu}\, \xi_{\Xi}\, \partial_\mu\otimes dx^{\nu\Xi}~,
\label{strans1x}
\eea
where $\xi\in \Omega^q$ and the multi-index $\Xi=\rho_1\dots \rho_q$.
It turns out that if $\xi$ is a $\hat\nabla$-covariantly constant  and $i_\xi dH=0$, i.e. it satisfies (\ref{invcon}),  one can show that the infinitesimal transformation
\bea
\delta_S X_\mu&=&\alpha_S \hat\nabla DX^\nu S_{\nu,\mu \Xi} DX^\Xi+{(-1)^q\over q+1} \hat\nabla(\alpha_S S_{\mu,\nu \Xi} DX^{\nu \Xi})
\cr
&-& \quad{2\over3 (q+1)^2 (q+2)} \alpha_S (H\wedge
\xi)_{\mu\nu\rho\Xi} DX^{\nu\rho \Xi}~,~~~
\label{strans1}
\eea
is a symmetry of the action.  Note that the proof of invariance of the action (\ref{act1}) under the transformations  (\ref{strans1})  requires the Bianchi identity (\ref{bianchi}).

\begin{theorem}\label{thlms}
If there exist forms $\sigma$ and $\xi$ such that the  identities
\bea
&&(L\cdot M)_{\mu L_2, \nu M_2}dx^\mu\otimes dx^{L_2\nu M_2}=(-1)^{\ell+1}  \sigma_{\mu\nu L_2 M_2}dx^{\mu\nu L_2 M_2}+{m\over2} g_{\mu\nu}  \xi_{L_2 M_2}dx^\mu\otimes dx^{\nu L_2 M_2}~,
\cr
&&(L\cdot M)_{\nu L_2, \mu M_2}dx^\mu\otimes dx^{\nu L_2 M_2}=(-1)^{\ell} \sigma_{\mu\nu L_2 M_2}dx^{\mu\nu L_2 M_2} +{\ell\over2} g_{\mu\nu}  \xi_{L_2 M_2}dx^\mu\otimes dx^{\nu L_2 M_2}~,
\cr
&&\Big((L\cdot M)_{\mu L_2, \nu M_2}+(\mu\leftrightarrow \nu)\Big)\,dx^{L_2 M_2}=g_{\mu\nu}  \xi_{L_2 M_2}dx^{L_2 M_2}
\cr
&&\qquad\qquad -{1\over2}(\ell+m-2) \Big(g_{\mu\mu_2} \xi_{\nu L_3  M_2}+(\mu\leftrightarrow \nu)\Big)dx^{\mu_2 L_3 M_2}~,
\label{conconx}
\eea
hold,
then the commutator $[\delta_L, \delta_M]$ (\ref{comm}) of the symmetries of the action (\ref{act1}) generated by $L$ and $M$ can be reorganised as
\bea
[\delta_L, \delta_M] X^\mu= \delta_{\vec \sigma} X^\mu+ \delta_{\vec\nu} X^\mu+ \delta_{ S} X^\mu~,
\eea
with each term in the right hand side of the equation to individually be a symmetry of the action (\ref{act1}).
The $\delta_{ \vec\sigma}$  transformation is  generated by the form ${\sigma}$ and similarly for $\delta_{\vec\nu} $ with
\bea
 \nu_{\mu L M} dx^{\mu L M}\defeq-(\ell+m+1)\Big[ H_{\nu\rho \mu} L^\nu{}_L M^\rho{}_{ M}\, dx^{\mu L M}+(-1)^\ell {\ell m\over 6}  H_{\mu\mu_1\mu_2} \xi_{L_3 M}\, dx^{\mu\mu_1\mu_2 L_3 M}\Big]~.
\eea
The latter form is a modification of the Nijenhuis tensor and satisfies $\hat\nabla\nu=0$.
Furthermore, $ S$ is constructed  from the $(\ell+m-2)$-form $ \xi$ and $ g$ as in (\ref{strans1x}).
\end{theorem}
\begin{proof}
The proof of this is similar to that given for the holonomy symmetries of 2-dimensional sigma models  in \cite{gpph}, which in turn is  a generalisation of earlier results presented in \cite{phgpw1, phgpw2}.
\end{proof}

\begin{remark}

The conditions (\ref{conconx}) described in the theorem above on $L\cdot M$ hold for the fundamental forms of many structure groups. In particular, they hold for the fundamental forms of the Berger groups.  However, there are also examples that these conditions do not hold, see \cite{eblggp}.  Further comments will be made  in the conclusions on how the theorem \ref{thlms} can be used to induce a (super) Lie algebra bracket on the space fundamental forms of a manifold with reduced structure group.

\end{remark}

\subsection{Non-compact holonomy groups}

To give an example of how  the commutator (\ref{comm}) can be used to induce a Lie algebra structure on $\mathfrak{f}^{\mathcal H}$ consider the heterotic (inspired)  geometries associated with a non-compact holonomy group. For this geometries, all fundamental forms are  $\hat\nabla$-covariantly constant and
null along a null $\hat\nabla$-covariantly constant 1-form $\kappa$, i.e. they are elements of $\Omega^*_{\kappa, \hat\nabla}$, see section \ref{sec:null}. For such forms, there is a simplification of the commutator (\ref{comm}) as follows.

\begin{theorem}\label{ncomcom}
Suppose that $\phi\in \Omega^{\ell+1}_{\kappa, \hat\nabla}$ and $\chi\in  \Omega^{m+1}_{\kappa, \hat\nabla}$,  the commutator (\ref{comm}) of two holonomy symmetries generated by the $\hat\nabla$-covariantly constant forms $\phi$ and $\chi$ is
\bea
[\delta_{\vec\phi}, \delta_{\vec\chi}] =\delta_\vkappa+\delta_{\overrightarrow{{\phi\bar{\curlywedge} \chi}}}~,
\label{walg}
\eea
with $a_\vkappa=-{\ell! m!\over (\ell+m-1)!} D(a_{\vec\phi} a_{\vec\chi} J_{\phi\bar{\curlywedge} \chi})$ and $a_{\overrightarrow{{\phi\bar{\curlywedge} \chi}}}=- {\ell! m!\over (\ell+m-2)!} a_{\vec\phi} a_{\vec\chi}\, DJ_\kappa$,   where the $\bar{\curlywedge}$ operation has been defined in theorem \ref{barcwedge}
\end{theorem}
\begin{proof}
The result follows after a direct computation of the right hand side of the commutator (\ref{comm}). The simplification is due to both the null property of the
forms that generate the symmetries as well as their invariance properties under the action of $\vkappa$.  Note that symmetry generated by $\kappa$ commutes with those generated by $\phi$ and $\chi$.
\end{proof}

The algebra of fundamental forms $\mathfrak{f}^{\mathcal H}_{\bar\wedge}$  of non-compact holonomy groups, like those in table \ref{tablehola},  is abelian.  This is because all fundamental forms are null along $\kappa$ . Alternatively, one can use   the commutator (\ref{walg}) and
consider $\mathfrak{f}^{\mathcal H}_{\bar\curlywedge}$ instead. As the operation $\bar\curlywedge$ satisfies the Jacobi identities and all the fundamental forms of the groups tabulated in \ref{tablehola}  have odd degree, $\mathfrak{f}^{\mathcal H}_{\bar\curlywedge}$  is a Lie algebra.  As the symmetry generated by $\kappa$ commutes with all the other symmetries generated by the remaining fundamental forms, let us consider $\mathring{\mathfrak{f}}^{\mathcal H}_{\bar\curlywedge}\defeq \mathfrak{f}^{\mathcal H}_{\bar\curlywedge}-\bR\langle \kappa\rangle$.  Clearly, the Lie algebra  $\mathfrak{f}^{\mathcal H}_{\bar\curlywedge}=\mathring{\mathfrak{f}}^{\mathcal H}_{\bar\curlywedge}\oplus\bR\langle \kappa\rangle$.

\begin{theorem}\label{ncomalg1}
The Lie algebras $\mathring{\mathfrak{f}}^{\mathcal H}_{\bar\curlywedge}$ of the fundamental forms of non-compact holonomy groups ${\mathcal H}$ tabulated in table \ref{tablehola} are given in table
\ref{tablefnulla}.
\end{theorem}

\begin{proof}
The proof of this result has been given in \cite{gpjp1}. Note that apart from $\mathring{\mathfrak{f}}^{\mathrm {Spin}(7)}_{\bar\curlywedge}$ and $\mathring{\mathfrak{f}}^{SU(4)}_{\bar\curlywedge}$, the fundamental forms of the rest of the groups are null 3-forms along $\kappa$.
$\mathring{\mathfrak{f}}^{SU(4)}_{\bar\curlywedge}$ is generated by the null forms $\omega$ and $\chi$ represented by the Hermitian and $(4,0)$ fundamental forms of $SU(4)$. Its closure requires the inclusion of $\curlywedge^3\omega$.  $\mathring{\mathfrak{f}}^{Sp(2)}_{\bar\curlywedge}$ is generated by the three null forms represented by  Hermitian forms of the $Sp(2)$ hyper-complex structure. A similar analysis leads to the identification of the remaining groups.
\end{proof}

\begin{table}[h]
 \begin{center}
\begin{tabular}{|c||c|c|c|c|c|c|c|}\hline
${\mathcal K}$&$\mathrm{Spin }(7)$ & $SU(4)$& $Sp(2)$&$\times^2Sp(1)$&$Sp(1)$&$U(1)$&$\{1\}$
 \\ \hline
 $\mathring{\mathfrak{f}}^{\mathcal H}_{\bar\curlywedge} $&\bR  & $\hat{\mathfrak{e}}(2) $ &$\mathfrak{sp}(1) $&$\oplus^2\mathfrak{sp}(1) $&$\mathfrak{so}(5) $&$\mathfrak{u}(4) $ &$\mathfrak{so}(8) $
\\ \hline
\end{tabular}
\end{center}
\vskip 0.2cm
\vskip 0.2cm
\caption{\label{tablefnulla}  {\small
In the first row, the ${\mathcal K}$ subalgebras of holonomy groups ${\mathcal H}={\mathcal K}\ltimes\bR^8$ of the supersymmetric heterotic backgrounds are stated. In the second row, the associated Lie algebras of the fundamental forms $\mathring{\mathfrak{f}}^{\mathcal H}_{\bar\curlywedge}$ are given. $\hat{\mathfrak{e}}(2) $ denotes the central extension of the Euclidean group in two dimensions.} }
\end{table}

\begin{remark}

Note that for a generic holonomy group, ${\mathcal H}$,   $\mathfrak{f}^{\mathcal H}_{\bar\curlywedge}$ defers from both $\mathfrak{f}^{\mathcal H}_{\bar\wedge}$ and $\mathfrak{f}^{\mathcal H}_{N}$.  We have already mentioned that $\mathfrak{f}^{\mathcal H}_{\bar\wedge}$ is always abelian. Also $\mathfrak{f}^{\mathcal H}_{N}$ depends on the choice of connection.  For example if the torsion $H$ vanishes, the Nijenhuis tensor of all the fundamental forms vanishes as well. In such a case,
$\mathfrak{f}^{\mathcal H}_{N}$ is abelian while  $\mathfrak{f}^{\mathcal H}_{\bar\curlywedge}$ will be given by the groups in table \ref{tablefnulla}.

\end{remark}

As the $\bar\curlywedge$ can be defined on null forms along $\kappa$ of any degree, one can also consider the (super)algebras $\mathfrak{f}^{\mathcal H}_{\bar\curlywedge}$, where ${\mathcal H}$ are described in section \ref{lorexn}. Observe that   $\kappa$ commutes with all the other fundamental forms.  So again, we consider $\mathring{\mathfrak{f}}^{\mathcal H}_{\bar\curlywedge}$ to simplify the description of the algebraic structure.

\begin{theorem}\label{ncomalg2}
The Lie (super)algebras $\mathring{\mathfrak{f}}^{\mathcal H}_{\bar\curlywedge}$ of the fundamental forms of non-compact holonomy groups, ${\mathcal K}\ltimes \bR^*$, for ${\mathcal K}= SO(n), U(n), SU(n), Sp(n)$ and $Sp(n)\cdot Sp(1)$  are given in table
\ref{tablefnullb}.
\end{theorem}

\begin{proof}
The result follows from a direct computation. In particular for ${\mathcal K}=SO(n)$,  $\mathring{\mathfrak{f}}_{\bar\curlywedge}=\bR\langle \epsilon\rangle$, where the $(n+1)$-form $\epsilon$ is represented the volume fundamental form $n$-form $\breve \epsilon$ of $SO(n)$ in directions transverse to the lightcone. Clearly, $\epsilon\bar\curlywedge  \epsilon=0$.

For ${\mathcal K}=U(n)$, $\mathring{\mathfrak{f}}_{\bar\curlywedge}=\bR\langle\omega\rangle$, where the 3-form $\omega$ is represented by the usual Hermitian 2-form $\breve\omega$ of  $U(n)$. Again, the Lie algebra structure on $\mathring{\mathfrak{f}}_{\bar\curlywedge}$ is abelian.

For ${\mathcal K}=SU(n)$, $\mathring{\mathfrak{f}}_{\bar\curlywedge}=\bR\langle  \omega, \chi_1, \chi_2\rangle$, where the 3-form $\omega$ is as in $U(n)$ case, and
$\chi_1$ and $\chi_2$ are $(n+1)$-forms represented by real and imaginary components of the  fundamental (n,0)-form  $\breve\chi$ of $SU(n)$. In particular, they are normalised as $\breve\chi_1=(\epsilon, \bar\epsilon)$ and $\breve\chi_2=(-i \epsilon, i \bar\epsilon)$.  The Lie algebra structure of  $\mathring{\mathfrak{f}}_{\bar\curlywedge}$ depends on whether $n$ is even or odd.  If $n=2k$, then  $\mathring{\mathfrak{f}}_{\bar\curlywedge}$ is a Lie algebra with non-vanishing commutation relations given by
\bea
&& \omega\bar\curlywedge  \chi_1=-n  \chi_2~,~~~ \omega\bar\curlywedge  \chi_2=n  \chi_1~,~~~
\cr
&& \chi_1\bar\curlywedge  \chi_2=-{2\over (n-1)!} \curlywedge^{n-1}  \omega~.
\eea
Note that this is the Euclidean algebra, $\hat{\mathfrak{e}}(2)$,  with a central extension  given by the generator $\curlywedge^{n-1}  \omega$.  This generator  is a non-minimal element which is required for the closure of the Lie algebra.

Next, if $n=2k+1$, then $\mathring{\mathfrak{f}}_{\bar\curlywedge}$ is a superalgebra with (anti)commutation relations
\bea
&& \omega\bar\curlywedge  \chi_1=-n  \chi_2~,~~~ \omega\bar\curlywedge  \chi_2=n  \chi_1~,
\cr
&& \chi_1\bar\curlywedge  \chi_1= \chi_2\bar\curlywedge  \chi_2={2\over (n-1)!} \curlywedge^{n-1}  \omega~.
\eea
 This is isomorphic to $N=2$ supersymmetry algebra, $\mathfrak{s}(2)$ in one dimension with $\chi_1$ and $\chi_2$ the supersymmetry generators, $\curlywedge^{n-1} \omega$ the hamiltonian and $\omega$ the R-symmetry generator that rotates  the two supersymmetry charges.

For ${\mathcal K}=Sp(n)$, one finds that $\mathring{\mathfrak{f}}_{\bar\curlywedge}=\mathfrak{sp}(1)$, i.e. it is the same Lie algebra as that we have stated for $n=2$.  Finally, for ${\mathcal K}=Sp(n)\cdot Sp(1)$, $\mathring{\mathfrak{f}}_{\bar\curlywedge}=\bR\langle \phi\rangle$, where the 5-form $\phi$ is represented by the fundamental 4-form of $Sp(n)\cdot Sp(1)$, i.e.  $\mathring{\mathfrak{f}}_{\bar\curlywedge}$ is abelian.

\end{proof}

\begin{table}[h]
 \begin{center}
\begin{tabular}{|c||c|c|c|c|c|c|c|}\hline
${\mathcal K}$&$SO(n)$&$U(n)$ & $SU(2k)$& $SU(2k+1)$& $Sp(n)$&$Sp(n)\cdot Sp(1)$
 \\ \hline
 $\mathring{\mathfrak{f}}^{\mathcal H}_{\bar\curlywedge} $&\bR  & $\bR $ &$\hat{\mathfrak{e}}(2) $&$\mathfrak{s}(2) $&$\mathfrak{sp}(1) $&$\bR $
\\ \hline
\end{tabular}
\end{center}
\vskip 0.2cm
\vskip 0.2cm
\caption{\label{tablefnullb}  {\small
In the first row, the ${\mathcal K}$ subalgebras of holonomy groups ${\mathcal H}={\mathcal K}\ltimes\bR^8$  are stated. In the second row, the associated Lie algebras, $\mathring{\mathfrak{f}}^{\mathcal H}_{\bar\curlywedge} $,  of the fundamental forms are given. } }
\end{table}

\begin{remark}

For completeness, the superalgebra $\mathring{\mathfrak{f}}^{\mathcal H}_{\bar\curlywedge}$, for ${\mathcal H}=G_2\ltimes\bR^7$, is given by the relations
$\varphi\bar\curlywedge\varphi=-\phi$ with the remaining commutators to vanish, where the 4-form $\varphi$ is represented with the fundamental $G_2$ 3-form $\breve\varphi$ and the 5-form $\phi$ is represented by the fundamental $G_2$ 4-form $\breve\phi$. The form $\breve\phi$  is the Hodge dual to $\breve\varphi$ in seven dimensions.
This superalgebra is isomorphic to $N=1$ supersymmetry algebra, $\mathfrak{s}(1)$,  in one dimension with $\varphi$ the supersymmetry generator and $-\phi$ the Hamiltonian generator.
\end{remark}

\section{Concluding remarks}

We have investigated some aspects of derivations on Euclidean and Lorentzian signature manifolds ${\mathcal M}^n$ that exhibit a reduction of the structure group to a subgroup ${\mathcal H}$ of the orthogonal group $SO_n$. In a variety of examples, we have identified  the Lie (super)algebra structure induced on the fundamental
forms of ${\mathcal H}$ by inner derivations. We have also pointed out that there is a close relationship between the investigation of holonomy symmetries in sigma models and the (super)algebra of inner and exterior derivations on manifolds. Guided by this, we have introduced a Lie (super)algebra operation $\bar\curlywedge$ on the space of null forms along a null 1-form $\kappa$ and identified the corresponding (super)algebras in a variety of examples.  These include the (super)algebras of heterotic geometries with non-compact holonomy groups as well as other geometries inspired by these heterotic structures.  We demonstrated that these superalgebras differ from both those induced by standard inner derivations as well as those induced by exterior derivations generated by the fundamental forms.

The extension of these results to exterior derivations on Euclidean  or Lorentzian signature manifolds with a compact holonomy group require further investigation. The main issue is that the Nijenhuis tensor of two ${\mathcal H}$-fundamental forms may not be $\nabla^{\mathcal H}$-covariantly constant, where $\nabla^{\mathcal H}$ is a connection with holonomy ${\mathcal H}$.  This potentially obstructs the closure of the (super)algebra of derivations in the space of
$\nabla^{\mathcal H}$-covariantly constant tensors. This issue can be resolved with a Gray-Hervella type of investigation to identify the classes which are compatible with the closure. It is expected that the (super)algebra  obtained will depend on the class of the underlying manifold. Some insight into the structures that can emerge in such an investigation, or  even for clues about extensions, can be seen in the exploration of the commutator of two holonomy symmetries in section \ref{sec:holsym}. It is apparent from the commutator that, in addition to the derivations associated with the fundamental forms of ${\mathcal H}$, one has to also include the exterior derivations constructed by wedging the identity vector 1-form with the fundamental forms of ${\mathcal H}$. As the derivations act on $\Omega^*({\mathcal M}^n)$, $\Omega^*({\mathcal M}^n)$ will decompose into representations of the (super)algebra of derivations. It will be of interest to understand the relationship between such decompositions and the geometric structure of the underlying manifold.






\end{document}